\documentclass[10pt]{amsart}
\usepackage{amssymb,amsbsy,amsmath,amsfonts,times, amscd}
\usepackage{latexsym,euscript,exscale}

\usepackage{helvet}

\title{Equivariant geometry of Banach spaces and topological groups}

\author {Christian Rosendal}
\address{Department of Mathematics, Statistics, and Computer Science (M/C 249)\\University of Illinois at Chicago\\851 S. Morgan St.\\Chicago, IL 60607-7045\\USA}
\email{rosendal.math@gmail.com}
\urladdr{http://homepages.math.uic.edu/$\sim$rosendal}

\date {}
\newcommand{\norm}[1]{\lVert#1\rVert}
\newcommand{\Norm}[1]{\big\lVert#1\big\rVert}
\newcommand{\NORM}[1]{\Big\lVert#1\Big\rVert}
\newcommand{\triple}[1]{|\!|\!|#1|\!|\!|}

\newcommand {\N}{\mathbb N}
\newcommand {\M}{\mathbb M}
\newcommand {\Q}{\mathbb Q}
\newcommand {\R}{\mathbb R}
\newcommand {\Z}{\mathbb Z}

\newcommand {\U}{\mathbb U}

\newcommand{\eps}{\epsilon}

\newcommand{\begr}{\!\upharpoonright}

\newcommand{\saa}{\Rightarrow}
\newcommand{\equi}{\Longleftrightarrow}

\newcommand{\til}{\rightarrow}

\newcommand {\Del}{ \; \Big| \;}
\newcommand {\del}{ \; \big| \;}

\newcommand {\go} {\mathfrak}
\newcommand {\ku} {\mathcal}

\newcommand{\ov}{\overline}
\newcommand{\inv}{^{-1}}

\newcommand {\e} {\exists}
\renewcommand {\a} {\forall}

\newtheorem{thm}{Theorem}
\newtheorem{cor}[thm]{Corollary}
\newtheorem{lemme}[thm]{Lemma}
\newtheorem{prop} [thm] {Proposition}
\newtheorem{defi} [thm] {Definition}

\newtheorem{prob}[thm]{Problem}
\newtheorem{quest}[thm]{Question}

\theoremstyle{definition}

\newtheorem{rem}[thm]{Remark}
\newtheorem{exa}[thm]{Example}

\usepackage[sc]{mathpazo}

\begin{document}

\thanks{The author was partially supported by a Simons Foundation Fellowship (Grant \#229959) and also recognises support from the NSF (DMS 1201295 \& DMS 1464974)}

\begin{abstract}
We study uniform and coarse embeddings between Banach spaces and topological groups. A particular focus is put on equivariant embeddings, i.e., continuous cocycles associated to continuous affine isometric actions of topological groups on separable Banach spaces with varying geometry. 
\end{abstract}

\maketitle

\tableofcontents


\section{Introduction}
The present paper is a contribution to the study of large scale geometry of Banach spaces and topological groups and, in particular, to questions of embeddability between these objects. In a sense, our aim is somewhat wider than usual as we will be dealing with general Polish groups as opposed to only locally compact groups. This is achieved by using the recently developed framework of \cite{rosendal-coarse}, which allows us to treat Banach spaces and topological groups under one heading. Still our focus will be restricted as we are mainly interested in equivariant maps, that is, cocycles associated to affine isometric actions on Banach spaces. 

To state the results of the paper, let us begin by fixing the basic terminology. 
Given a map $\sigma\colon (X,d)\til (Y,\partial)$  between metric spaces, we define the {\em compression modulus} by
$$
\kappa_\sigma(t)=\inf\big(\partial(\sigma(a),\sigma(b))\mid d(a,b)\geqslant t\big)
$$
and the {\em expansion modulus} by 
$$
\theta_\sigma(t)=\sup\big(\partial(\sigma(a),\sigma(b))\mid d(a,b)\leqslant t\big).
$$
Thus $\sigma$ is {\em uniformly continuous} if $\lim_{t\til 0_+}\theta_\sigma(t)=0$ (in which case $\theta_\sigma$ becomes the modulus of uniform continuity) and a {\em uniform embedding} if, moreover, $\kappa_\sigma(t)>0$ for all $t>0$. Furthermore, $\sigma$ is {\em bornologous} if $\theta_\sigma(t)<\infty$ for all $t<\infty$ and {\em expanding} if $\lim_{t\til \infty}\kappa_\sigma(t)=\infty$. A bornologous expanding map is called a {\em coarse embedding}.
We also define $\sigma$ to be  {\em uncollapsed} if $\kappa_\sigma(t)>0$ for just some sufficiently large $t>0$. 

As we will be studying uniform and coarse embeddability between Banach spaces and topological groups, we must extend the above concepts to the larger categories of uniform and coarse spaces and also show how every topological group is canonically equipped with both a uniform and a coarse structure. 

Postponing for the moment this discussion, let consider the outcomes of our study.
As pointed out by N. Kalton \cite{kalton2}, the concepts of uniform and coarse embeddability between Banach spaces seem very tightly related. Though, Kalton \cite{kalton3} eventually was able to give an example of two separable Banach spaces that are coarsely equivalent, but not uniformly homeomorphic, the following basic question of Kalton concerning embeddings  remains open.
\begin{quest}
Does the following equivalence hold for all (separable) Banach spaces?
$$
X \text{ is uniformly embeddable into }E\;\equi\; X \text{ is coarsely embeddable into }E.
$$
\end{quest}

Relying on entirely elementary techniques, our first result shows that in many settings we do have an implication from left to right.
\begin{thm}\label{intro:ell p sum}
Suppose $\sigma\colon X\til E$ is an uncollapsed uniformly continuous map between Banach spaces. Then, for any $1\leqslant p<\infty$,  $X$ admits a simultaneously uniform and coarse embedding  into $\ell^p(E)$.
\end{thm}
Since both uniform and coarse embeddings are uncollapsed, we see that if $X$ is uniformly embeddable into $E$, then $X$ is coarsely embeddable into $\ell^p(E)$. For the other direction, if $X$ admits a uniformly continuous coarse embedding into $E$, then $X$ is uniformly embeddable into $\ell^p(E)$.
It is therefore natural to ask to which extent bornologous maps can be replaced by uniformly continuous maps. In particular, is every bornologous map $\sigma\colon X\til E$  between Banach spaces {\em close} to a uniformly continuous map, i.e., is there a uniformly continuous map $\tilde \sigma$ so that $\sup_{x\in X}\norm{\sigma(x)-\tilde\sigma(x)}<\infty$?

As it turns out, A. Naor \cite{naor-nets} was recently able to answer our question in the negative, namely, there are separable Banach spaces $X$ and $E$ and a bornologous map between them which is not close to any uniformly continuous map. 
Nevertheless, several weaker questions remain open.

Also, in case the passage from $E$ to $\ell^p(E)$ proves troublesome, we can get by with $E\oplus E$ under stronger assumptions.
\begin{thm}\label{intro:E plus E}
Suppose $\sigma\colon X\til B_E$ is an uncollapsed uniformly continuous map from a Banach space $X$ into the ball of a Banach space $E$. Then $X$ admits a uniformly continuous coarse embedding into $E\oplus E$.
\end{thm}
For example, if $X$ is a Banach space uniformly embeddable into its unit ball $B_X$, e.g., if $X=\ell^2$, then, whenever $X$ uniformly embeds into a Banach space $E$, it coarsely embeds into $E\oplus E$.

The main aim of the present paper however is to consider equivariant embeddings between topological groups and Banach spaces, i.e., continuous cocycles. So let $\pi\colon G\curvearrowright E$ be a strongly continuous isometric linear representation of a topological group $G$ on a Banach space $E$, i.e., each $\pi(g)$ is a linear isometry of $E$ and, for every $\xi\in E$, the map $g\in G\mapsto \pi(g)\xi$ is continuous. A continuous {\em cocycle} associated to $\pi$ is a continuous map $b\colon G\til E$ satisfying the cocycle equation
$$
b(gf)=\pi(g)b(f)+b(g).
$$
This corresponds to the requirement that $\alpha(g)\xi=\pi(g)\xi+b(g)$ defines a continuous action of $G$ by affine isometries on $E$. As $b$ is simply the orbit map $g\mapsto \alpha(g)0$, it follows that continuous cocycles are actually uniformly continuous and bornologous. 
We call  a continuous cocycle $b\colon G\til E$  {\em coarsely proper} if it is a coarse embedding of $G$ into $E$. In this case, we also say that the associated affine isometric action $\alpha$ is {\em coarsely proper}.

For the next result,  a topological group $G$ has the {\em Haagerup property} if it has a strongly continuous unitary representation with an associated coarsely proper cocycle. Though it is known that every locally compact second countable amenable group has the Haagerup property \cite{BCV}, this is very far from being true for general topological groups. Indeed, even for Polish groups, that is, separable completely metrisable topological groups, such as separable Banach spaces, this is a significant requirement.

Building on work of I. Aharoni, B. Maurey and B. S. Mityagin \cite{maurey}, we show the following equivalence.
\begin{thm}\label{intro:haagerup equiv}
The following conditions are equivalent for an amenable Polish group $G$,
\begin{enumerate}
\item $G$ coarsely embeds into a Hilbert space,
\item $G$ has the Haagerup property.
\end{enumerate}
\end{thm}

Restricting to Banach spaces, we have a stronger result, where the equivalence of (1) and (2) is due to N. L. Randrianarivony \cite{randrianarivony2}.
\begin{thm}\label{intro:haagerup banach}
The following conditions are equivalent for a separable Banach space $X$,
\begin{enumerate}
\item $X$ coarsely embeds into a Hilbert space,
\item $X$ uniformly embeds into a Hilbert space,
\item $X$ admits an uncollapsed uniformly continuous map into a Hilbert space,
\item $X$ has the Haagerup property.
\end{enumerate}
\end{thm}

Even from collapsed maps we may obtain information, provided that there is just some single distance not entirely collapsed.
\begin{thm}
Suppose $\sigma\colon X\til \ku H$ is a uniformly continuous map from a separable Banach space into Hilbert space so that, for some single $r>0$,
$$
\inf_{\norm{x-y}=r}\norm{\sigma(x)-\sigma(y)}>0.
$$
Then $B_X$ uniformly embeds into $B_\ku H$.
\end{thm}

One of the motivations for studying cocycles, as opposed to general uniform or coarse embeddings, is that cocycles  are also algebraic maps, i.e., reflect algebraic features of the acting group $G$. As such, they have a higher degree of regularity and permit us to carry geometric information from the phase space back to the acting group. For example, a continuous uncollapsed cocycle between Banach spaces is automatically a uniform embedding. Similarly, a cocycle $b\colon X\til E$ between Banach spaces associated to the trivial representation $\pi\equiv {\rm id}_E$ is simply a bounded linear operator, so a cocycle may be considered second best to a linear operator. We shall encounter and exploit many more instances of this added regularity throughout the paper.

Relaxing the geometric restrictions on the phase space, the next case to consider is that of super-reflexive spaces, i.e., spaces admitting a uniformly convex renorming. For this class, earlier work was done by V. Pestov \cite{pestov} and Naor--Y. Peres \cite{naor-peres} for discrete amenable groups $G$. For topological groups, several severe obstructions appear and it does not seem possible to get an exact analogue of Theorem \ref{intro:haagerup equiv}. Indeed, the very concept of amenability requires reexamination. We say that a Polish group $G$ is {\em F\o lner amenable} if there is either a continuous homomorphism $\phi\colon H\til G$ from a locally compact second countable amenable group $H$ with dense image or if $G$ admits a chain of compact subgroups with dense union. For example, a separable Banach space is F\o lner amenable.

\begin{thm}\label{intro:fin repr}
Let $G$ be a F\o lner amenable Polish group admitting a uniformly continuous coarse embedding into a Banach space $E$. Then, for every $1\leqslant p<\infty$, $G$ admits a coarsely proper continuous affine isometric action on a Banach space $V$ that is finitely representable in $L^p(E)$.
\end{thm}

We note that, if $E$ is super-reflexive, then so is every $V$ finitely representable in  $L^2(E)$. 
Again, for Banach spaces, this leads to the following.

\begin{thm}\label{intro:pestov banach}
Suppose $X$ is a separable Banach space admitting an uncollapsed  uniformly continuous map into a super-reflexive space. Then $X$ admits a coarsely proper continuous cocycle with values in a super-reflexive space.
\end{thm}

For our next results, we define a topological group to be {\em metrically stable} if it admits a compatible left-invariant stable metric. By results of \cite{ben yaacov, megrelishvili2, shtern}, a Polish group is metrically stable if and only if it is isomorphic to a subgroup if the linear isometry group of a separable reflexive Banach space under the strong operator topology. Also, by a result of Y. Raynaud \cite{raynaud}, a metrically stable Banach space contains a copy of $\ell^p$ for some $1\leqslant p<\infty$.

Using the construction underlying Theorem \ref{intro:fin repr}, we obtain information on spaces uniformly embeddable into balls of super-reflexive spaces.
\begin{thm}\label{intro:embedding super-reflexive}
Let $X$ be a Banach space admitting an uncollapsed uniformly continuous map into the unit ball $B_E$ of a super-reflexive Banach space $E$. Then $X$ is metrically stable and contains an isomorphic copy of $\ell^p$ for some $1\leqslant p<\infty$. 
\end{thm}
As an application, this restricts the class of super-reflexive spaces uniformly embeddable into their own balls.

For a general non-amenable topological group $G$, we may also produce coarsely proper affine isometric actions on reflexive spaces starting directly from a stable metric or \'ecart. 
\begin{thm}\label{intro:stable metric refl}
Suppose a topological group $G$ carries a continuous left-invariant coarsely proper stable \'ecart. Then $G$ admits a coarsely proper continuous affine isometric action on a reflexive Banach space.
\end{thm}

Intended applications here are, for example, automorphisms groups of countable atomic models of countable stable first-order theories that, under the additional assumption of being locally (OB), satisfy the conditions of Theorem \ref{intro:stable metric refl}.

The final part of our investigations concern a fixed point property for affine isometric group actions. We identify  a geometric incompatibility between a topological group $G$ and a Banach space $E$ strong enough to ensure that not only does $G$ have no coarsely proper affine isometric action on $E$, but every affine isometric action even has a fixed point. 

The two main concepts here are solvent maps and geometric Gelfand pairs. First, a map $\phi\colon X\til Y$ between metric spaces is {\em solvent} if, for every $n$, there is an $R$ with  $R\leqslant d(x,x')\leqslant R+n\saa d(\phi x, \phi x')\geqslant n$. Refining earlier results of Kalton, we show that every bornologous map from $c_0$ to a reflexive Banach space is insolvent.  Also a coarsely proper continuous isometric action $G\curvearrowright X$ of a topological group $G$ on a metric space $X$ is said to be a {\em geometric Gelfand pair} if, for some $K$ and all $x,y,z,u\in X$ with $d(x,y)\leqslant d(z,u)$, there is $g\in G$ so that $d(g(x),z)\leqslant K$ and $d(z,g(y))+d(g(y),u)\leqslant d(z,u)+K$. This second condition is typically verified when $X$ is sufficiently geodesic and the action of $G$ is almost doubly transitive. For example, if ${\rm Aff}(X)$ denotes the group of affine isometries of a Banach space $X$, then ${\rm Aff}(X)\curvearrowright X$ is a geometric Gelfand pair when $X=L^p([0,1])$, $1\leqslant p<\infty$, and when $X$ is the Gurarii space. Similarly, if $X$ is the integral, $\Z\U$, or rational, $\Q\U$, Urysohn metric space, then ${\rm Isom}(X)\curvearrowright X$ is a geometric Gelfand pair. 

\begin{thm}\label{intro:bdd orbits}
Suppose $G\curvearrowright X$ is a geometric Gelfand pair and $Y$ is a metric space so that every bornologous map $X\til Y$ is insolvent. Then every continuous isometric action $G\curvearrowright Y$ has bounded orbits.
\end{thm}

Combining Theorem \ref{intro:bdd orbits} with the observations above and the fixed point theorems of \cite{ryll} and \cite{monod}, we obtain the following corollary. 
\begin{cor}\label{intro:fixed point}
Every continuous affine isometric action of ${\rm Isom}(\Q\U)$ on a reflexive Banach space  or on $L^1([0,1])$ has a fixed point.
\end{cor}

\begin{center}{ Acknowledgements}\end{center}
I would like to thank a number of people for helpful conversations and other aid during the preparation of the paper. They include F. Baudier, B. Braga, J. Galindo, G. Godefroy,  W. B. Johnson,  N. Monod,  A. Naor, Th. Schlumprecht and A. Thom.


\section{Uniform and coarse structures on topological groups}\label{coarse defi}
As is well-known, the non-linear geometry of Banach spaces and large scale geometry of finitely generated groups share many common concepts and tools, while a priori dealing with distinct subject matters. However, as shown in \cite{rosendal-coarse}, both theories may be viewed as instances of the same overarching framework, namely, the coarse geometry of topological groups. Thus, many results or problems admitting analogous but separate treatments for Banach spaces and groups can in fact be entered into this unified framework. 

Recall that, if $G$ is a topological group, the {\em left-uniform structure} $\ku U_L$ on $G$ is the uniform structure generated by the family of entourages
$$
E_V=\{(x,y)\in G\times G\mid x\inv y\in V\},
$$
where $V$ varies over identity neighbourhoods in $G$. It is a  fact due to A. Weil that the left-uniform structure on $G$ is given as the union $\ku U_L=\bigcup_d\ku U_d$ of the uniformities $\ku U_d$ induced by continuous left-invariant  \'ecarts (aka. pseudometrics) $d$ on $G$. Thus, $E\in \ku U_L$ if it contains some
$$
E_\alpha=\{(x,y)\in X\times X\del d(x,y)<\alpha\},
$$
with $\alpha>0$ and $d$ a continuous left-invariant \'ecart on $G$. Apart from the weak-uniformity on a Banach space, this is the only uniformity on $G$ that we will consider and, in the case of the additive group $(X,+)$ of a Banach space, is simply the uniformity given by the norm.

More recently, in \cite{rosendal-coarse} we have developed a theory of coarse geometry of topological groups the basic concepts of which are analogous to those of the left-uniformity. For this, we first need J. Roe's concept of a coarse space \cite{roe}. A {\em coarse structure}  on a set $X$ is a family $\ku E$ of subsets $E\subseteq X\times X$ called {\em coarse entourages} satisfying 
\begin{itemize}
\item[(i)] the diagonal $\Delta$ belongs to $\ku E$,
\item[(ii)] $F\subseteq E\in \ku E\saa F\in \ku E$, 
\item[(iii)] $E\in \ku E\saa E\inv=\{(y,x)\del (x,y)\in E\}\in \ku E$, 
\item[(iv)] $E,F\in \ku E\saa E\cup F\in \ku E$, 
\item[(v)] $E,F\in \ku E\saa E\circ F=\{(x,y)\del \e z\; (x,z)\in E\;\&\; (z,y)\in F\}\in \ku E$.
\end{itemize}
Just as the prime example of a uniform space is a metric space, the motivating example of a coarse space $(X,\ku E)$ is that induced from a (pseudo) metric space $(X,d)$. Indeed, in this case, we let $E\in \ku E_d$ if $E$ is contained in some 
$$
E_\alpha=\{(x,y)\in X\times X\del d(x,y)<\alpha\},
$$
with $\alpha<\infty$.

Now, if $G$ is a topological group, we define the left-coarse structure $\ku E_L$ to be the coarse structure given by
$$
\ku E_L=\bigcap_d\ku E_d,
$$
where the intersection is taken over the family of coarse structures $\ku E_d$ given by continuous left-invariant \'ecarts on $G$. Again, this is the only coarse structure on a topological group we will consider and, in the case of a finitely generated or locally compact, compactly generated group, coincides with the coarse structure given by the (left-invariant) word metric. Similarly, for a Banach space, the coarse structure $\ku E_L$ is simply that given by the norm. 

A subset $A$ of a topological group has {\em property (OB) relative to $G$} if $A$ has finite diameter with respect to every continuous left-invariant \'ecart on $G$. Also, $G$ has {\em property (OB)} if is has property (OB) relative to itself. 
A continuous left-invariant \'ecart $d$ on a topological group $G$ is said to be {\em coarsely proper} if it induces the coarse structure, i.e., if $\ku E_L=\ku E_d$. In the class of Polish groups, having a coarsely proper \'ecart is equivalent to the group being {\em locally (OB)}, that is, having a relatively (OB) identity neighbourhood. Finally, $G$ is {\em (OB) generated} if it is generated by a relatively (OB) set.


Recall that a map $\sigma\colon (X,\ku U)\til (Y,\ku V)$ between uniform spaces is {\em uniformly continuous} if, for every $F\in \ku V$, there is $E\in \ku U$ so that $(a,b)\in E\saa (\sigma(a),\sigma(b))\in F$. Also, $\sigma$ is a {\em uniform embedding} is, moreover, for every $E\in \ku U$, there is $F\in \ku V$ so that $(a,b)\notin E\saa (\sigma(a),\sigma(b))\notin F$. Similarly, a map $\sigma\colon (X,\ku E)\til (Y,\ku F)$ between coarse spaces is {\em bornologous} if, for every $E\in \ku E$, there is $F\in \ku F$ so that $(a,b)\in E\saa (\sigma(a),\sigma(b))\in F$. And $\sigma$ is {\em expanding} if, every $F\in \ku F$, there is $E\in \ku E$ so that $(a,b)\notin E\saa (\sigma(a),\sigma(b))\notin F$. An expanding bornologous map is called a {\em coarse embedding}.  

A coarse embedding $\sigma$ is a {\em coarse equivalence} if, moreover, the image $\sigma[X]$ is {\em cobounded} in $Y$, that is, there is some $F\in \ku F$ so that
$$
\a y\in Y\; \e x\in X\; (y,\sigma(x))\in F.
$$

A map $\sigma\colon (X,\ku E)\til (Y,\ku U)$ from a coarse space $X$ to a uniform space $Y$ is {\em uncollapsed} if there are a coarse entourage $E\in \ku E$ and a uniform entourage $F\in \ku U$ so that
$$
(a,b)\notin E\;\saa\; (\sigma(a),\sigma(b))\notin F.
$$
These definitions all agree with those given for the specific case of metric  spaces.

It turns out to be useful to introduce a finer modulus than the compression. For this, given a map $\sigma\colon (X,d)\til (Y,\partial)$ between metric spaces, define the {\em exact compression modulus} by
$$
\tilde\kappa_\sigma(t)=\inf\big(\partial(\sigma(a),\sigma(b))\mid d(a,b)= t\big)
$$
and observe that  $ \kappa_\sigma(t)=\inf_{s\geqslant t}\tilde\kappa_\sigma(s)$.



\section{Uniform versus coarse embeddings between Banach spaces}\label{uniform vs coarse}
Observe first that a map $\sigma\colon X\til M$ from a Banach space $X$ to a uniform space $(M, \ku U)$ is uncollapsed if there are  $\Delta>0$ and an entourage $F\in \ku U$ so that
$$
\norm {x-y}>\Delta\;\saa\; (\sigma(x),\sigma(y))\notin F.
$$
For example, a uniform embedding is uncollapsed. Note also  that, if $\sigma\colon X\til G$ is a uniformly continuous  uncollapsed map into a topological group $G$, then $x\mapsto (\sigma(x), \sigma(2x), \sigma(3x), \ldots)$ defines a uniform embedding of $X$ into the infinite product $\prod_{n\in \N}G$.

Our main results about uncollapsed maps between Banach spaces are as follows.
\begin{thm}\label{ell p sum}
Suppose $\sigma\colon X\til E$ is an uncollapsed uniformly continuous map between Banach spaces. Then, for any $1\leqslant p<\infty$,  $X$ admits a simultaneously uniform and coarse embedding  into $\ell^p(E)$.
\end{thm}

The next corollary then  follows from observing that all of the classes of spaces  listed are closed under the operation $E\mapsto \ell^p(E)$ for an appropriate $1\leqslant p<\infty$.
\begin{cor}\label{cor to ell p sum}
If a Banach space $X$ is uniformly embeddable into $\ell^p$, $L^p$ (for some $1\leqslant p<\infty$), a reflexive, super-reflexive, stable, super-stable, non-trivial type or cotype space, then $X$ admits a simultaneously uniform and coarse embedding into a space of the same kind.
\end{cor}

If we wish to avoid the passage from $E$ to the infinite sum $\ell^p(E)$, we can get by with a direct sum $E\oplus E$, but only assuming that $\sigma$ maps into a bounded set. Also, the resulting map may no longer be a uniform embedding.

\begin{thm}\label{E plus E}
Suppose $\sigma\colon X\til B_E$ is an uncollapsed uniformly continuous map from a Banach space $X$ into the ball of a Banach space $E$. Then $X$ admits a uniformly continuous coarse embedding into $E\oplus E$.
\end{thm}

Both propositions will be consequences of somewhat finer and more detailed results with wider applicability. Indeed, Theorem \ref{ell p sum} is a direct corollary of Lemma \ref{unif saa coarse} below.

\begin{lemme}\label{unif saa coarse}
Suppose $X$ and $E$ are Banach spaces and $P_n\colon E\til E$ is a sequence of bounded projections onto subspaces $E_n\subseteq E$ so that,  $E_m\subseteq \ker P_n$ for all $m\neq n$. Assume also that $\sigma_n\colon X\til E_n$ are uncollapsed uniformly continuous maps. Then $X$ admits a simultaneously uniform and coarse embedding  into $E$.
\end{lemme}
This lemma applies in particular to the case when $E$ is a Schauder sum of a sequence of subspaces $E_n$.

\begin{proof}
By composing with a translation, we may suppose that $\sigma_n(0)=0$ for each $n$. Fix also $\Delta_n,\delta_n,\eps_n>0$ so that
$$
\norm{x-y}\geqslant \Delta_n\;\saa\; \norm{\sigma_n(x)-\sigma_n(y)}\geqslant \delta_n
$$
and
$$
\norm{x-y}\leqslant \eps_n\;\saa\; \norm{\sigma_n(x)-\sigma_n(y)}\leqslant 2^{-n}.
$$
Note that, if $\norm{x-y}\leqslant k\cdot\eps_n$ for some $k\in \N$, then there are $z_0=x, z_1, z_2, \ldots, z_k=y\in X$ so that $\norm{z_i-z_{i+1}}=\frac 1k\norm{x-y}\leqslant \eps_n$, whence
$$
\norm{\sigma_n(x)-\sigma_n(y)}\leqslant \sum_{i=1}^k\norm{\sigma_n(z_{i-1})-\sigma_n(z_i)}\leqslant k\cdot 2^{-n}.
$$
Thus, setting $\psi_n(x)=\frac{\sigma_n(n\Delta_n\cdot x)}{\lceil\frac{n^2\Delta_n}{\eps_n}\rceil}$, we have, for all $x,y\in X$, 
\[\begin{split}
\norm{x-y}\leqslant n
&\;\saa\; \norm{n\Delta_n\cdot x-n\Delta_n\cdot y}\leqslant n^2\Delta_n<\Big\lceil\frac{n^2\Delta_n}{\eps_n}\Big\rceil\cdot \eps_n\\
&\;\saa\; \norm{\sigma_n\big(n\Delta_n\cdot x\big)-\sigma_n\big(n\Delta_n\cdot y\big)}\leqslant \Big\lceil\frac{n^2\Delta_n}{\eps_n}\Big\rceil\cdot 2^{-n}\\
&\;\saa\; \norm{\psi_n(x)-\psi_n(y)}\leqslant 2^{-n},
\end{split}\]
while 
\[\begin{split}
\norm{x-y}\geqslant \frac 1n
&\;\saa\; \norm{n\Delta_n\cdot x-n\Delta_n\cdot y}\geqslant \Delta_n\\
&\;\saa\; \norm{\sigma_n\big(n\Delta_n\cdot x\big)-\sigma_n\big(n\Delta_n\cdot y\big)}\geqslant \delta_n\\
&\;\saa\; \norm{\psi_n(x)-\psi_n(y)}\geqslant \frac {\delta_n}{\big\lceil\frac{n^2\Delta_n}{\eps_n}\big\rceil}.
\end{split}\]


Now choose $\xi_n>0$ so that
$$
\norm{x-y}\leqslant \xi_n\;\saa\; \norm{\sigma_n(x)-\sigma_n(y)}\leqslant \frac{\delta_n}{n\norm{P_n}}\cdot2^{-n}
$$
and set $\phi_n(x)=\frac{n\norm{P_n}}{\delta_n}\cdot\sigma_n\big(\frac{\xi_n}n\cdot x\big)$. Then
\[\begin{split}
\norm{x-y}\leqslant n
&\;\saa\; \norm{\frac {\xi_n}n\cdot x-\frac {\xi_n}n\cdot y}\leqslant \xi_n\\
&\;\saa\; \Norm{\sigma_n\big(\frac {\xi_n}n\cdot x\big)-\sigma_n\big(\frac {\xi_n}n\cdot y\big)}\leqslant \frac{\delta_n}{n\norm{P_n}}\cdot2^{-n}\\
&\;\saa\; \norm{\phi_n(x)-\phi_n(y)}\leqslant 2^{-n},
\end{split}\]
while 
\[\begin{split}
\norm{x-y}\geqslant \frac{n\Delta_n}{\xi_n}
&\;\saa\; \norm{\frac {\xi_n}n\cdot x-\frac {\xi_n}n\cdot y}\geqslant \Delta_n\\
&\;\saa\;\Norm{\sigma_n\big(\frac {\xi_n}n\cdot x\big)-\sigma_n\big(\frac {\xi_n}n\cdot y\big)}\geqslant \delta_n\\
&\;\saa\; \norm{\phi_n(x)-\phi_n(y)}\geqslant n\norm{P_n}.
\end{split}\]

In particular, if $\norm{x-y}\leqslant m$, then $\norm{x-y}\leqslant n$ for all $n\geqslant m$, whence
$$
\sum_{n=1}^\infty\norm{\psi_{2n-1}(x)-\psi_{2n-1}(y)}\leqslant \sum_{n=1}^{m-1}\norm{\psi_{2n-1}(x)-\psi_{2n-1}(y)}+\sum_{n=m}^\infty2^{-2n+1}<\infty
$$
and
$$
\sum_{n=1}^\infty\norm{\phi_{2n}(x)-\phi_{2n}(y)}\leqslant \sum_{n=1}^{m-1}\norm{\phi_{2n}(x)-\phi_{2n}(y)}+\sum_{n=m}^\infty2^{-2n}<\infty.
$$
Setting $y=0$, we see that both $\sum_{n=1}^\infty\psi_{2n-1}(x)$ and $\sum_{n=1}^\infty\phi_{2n}(x)$ are absolutely convergent in $E$, whence we may define $\omega\colon X\til E$ by 
$$
\omega(x)=\sum_{n=1}^\infty\psi_{2n-1}(x)+\sum_{n=1}^\infty\phi_{2n}(x).
$$

First, to see that $\omega$ is uniformly continuous and thus bornologous, let $\eps>0$ and find $m$ large enough so that $2^{-2m+2}<\frac\eps3$. Since each of $\sigma_n$ is uniformly continuous, so are the $\psi_n$ and $\phi_n$. We may therefore choose $\eta>0$ so that 
$$
\NORM{\Big(\sum_{n=1}^{m-1}\psi_{2n-1}(x)+\sum_{n=1}^{m-1}\phi_{2n}(x)\Big)-\Big(\sum_{n=1}^{m-1}\psi_{2n-1}(y)+\sum_{n=1}^{m-1}\phi_{2n}(y)\Big)}<\frac \eps3
$$
whenever $\norm{x-y}<\eta$. Thus, if $\norm{x-y}<\min\{\eta, m\}$, we have 
\[\begin{split}
\norm{\omega(x)-\omega(y)}
\leqslant &\NORM{\Big(\sum_{n=1}^{m-1}\psi_{2n-1}(x)+\sum_{n=1}^{m-1}\phi_{2n}(x)\Big)-\Big(\sum_{n=1}^{m-1}\psi_{2n-1}(y)+\sum_{n=1}^{m-1}\phi_{2n}(y)\Big)}\\
&+\sum_{n=m}^\infty\Norm{\psi_{2n-1}(x)-\psi_{2n-1}(y)}+\sum_{n=m}^\infty\Norm{\phi_{2n}(x)-\phi_{2n}(y)}\\
<& \frac \eps3+\frac \eps3+\frac \eps3,
\end{split}\]
showing uniform continuity.

Secondly, to see that $\omega$ is a uniform embedding, suppose that $\norm{x-y}>\frac 1{2n-1}$ for some $n\geqslant 1$. Then,
\[\begin{split}
\norm{\omega(x)-\omega(y)}
&\geqslant\frac1{ \norm{P_{2n-1}}}\norm{P_{2n-1}\omega(x)-P_{2n-1}\omega(y)}\\
&\geqslant\frac1{ \norm{P_{2n-1}}}\norm{\psi_{2n-1}(x)-\psi_{2n-1}(y)}\\
&\geqslant \frac 1 {\norm{P_{2n-1}}}\cdot \frac {\delta_{2n-1}}{\big\lceil\frac{(2n-1)^2\Delta_{2n-1}}{\eps_{2n-1}}\big\rceil}.
\end{split}\]

Finally, to see that $\omega$ is a coarse embedding, observe  that, if $\norm{x-y}\geqslant \frac{2n\Delta_{2n}}{\xi_{2n}}$, then
\[\begin{split}
\norm{\omega(x)-\omega(y)}
&\geqslant\frac1{ \norm{P_{2n}}}\norm{P_{2n}\omega(x)-P_{2n}\omega(y)}\\
&\geqslant\frac1{ \norm{P_{2n}}}\norm{\phi_{2n}(x)-\phi_{2n}(y)}\\
&\geqslant 2n,
\end{split}\]
which finishes the proof.
\end{proof}

We should mention here that B. Braga \cite{braga1} has been able to use our construction above coupled with a result of E. Odell and T. Schlumprecht \cite{distortion} to show that $\ell^2$ admits a simultaneously uniform and coarse embedding into every Banach space with an unconditional basis and finite cotype.


Our next result immediately implies theorem \ref{E plus E}.
\begin{lemme}\label{unif-coarse}
Suppose $\sigma\colon X\til B_E$ and $\omega\colon X\til B_F$ are uncollapsed uniformly continuous maps from a Banach space $X$ into the balls of Banach spaces $E$ and $F$.  Then $X$ admits  uniformly continuous coarse embedding into $E\oplus F$. 
\end{lemme}

\begin{proof}
Since $\sigma$ and $\omega$ are uncollapsed, pick $\Delta\geqslant 2$ and $\delta>0$ so that 
$$
\norm {x-y}>\Delta\;\saa\; \norm{\sigma(x)-\sigma(y)}>\delta \;\;\&\;\; \norm{\omega(x)-\omega(y)}>\delta.
$$
We will inductively define bounded uniformly continuous maps
$$
\phi_1,\phi_2,\ldots\colon X\til E
$$
and
$$
\psi_1,\psi_2,\ldots\colon Y\til F
$$
with $\phi_n(0)=\psi_n(0)=0$ and numbers $0=t_0<r_1<t_1<r_2<t_2<\ldots$ with $\lim_{n}r_n=\infty$ so that, for all $n\geqslant 1$,
$$
\norm{x-y}\geqslant r_n\;\saa\; \Norm{\sum_{i=1}^n\phi_i(x)- \sum_{i=1}^n\phi_i(y)}\geqslant 2^n,
$$
$$
\norm{x-y}\geqslant t_n\;\saa\; \Norm{\sum_{i=1}^n\psi_i(x)- \sum_{i=1}^n\psi_i(y)}\geqslant 2^n,
$$
$$
\norm{x-y}\leqslant t_{n-1}\;\saa\; \Norm{\phi_n(x)- \phi_n(y)}\leqslant 2^{-n}
$$
and
$$
\norm{x-y}\leqslant r_n\;\saa\; \Norm{\psi_n(x)- \psi_n(y)}\leqslant 2^{-n}.
$$

Suppose that this have been done. Then
$$
\norm{x-y}\leqslant t_{n-1}\;\;\saa\;\; \sum_{i=n}^\infty\Norm{\phi_i(x)- \phi_i(y)}\leqslant \sum_{i=n}^\infty2^{-i}\leqslant 1.
$$
In particular, setting $y=0$, we see that the series $\sum_{i=1}^\infty\phi_i(x)$ is absolutely convergent for all $x\in X$. Similarly, 
$$
\norm{x-y}\leqslant r_{n}\;\saa\; \sum_{i=n}^\infty\Norm{\psi_i(x)- \psi_i(y)}\leqslant \sum_{i=n}^\infty2^{-i}\leqslant 1,
$$
showing that also $\sum_{i=1}^\infty\psi_i(x)$ is absolutely convergent for all $x\in X$.
So define $\phi\colon X\til E$ and $\psi\colon X\til F$ by $\phi(x)=\sum_{i=1}^\infty\phi_i(x)$ and $\psi(x)=\sum_{i=1}^\infty\psi_i(x)$.

We claim that $\phi$ and $\psi$ are uniformly continuous. To see this, let $\alpha>0$ be given and pick $n\geqslant 2$ so that $2^{-n+2}<\alpha$. By uniform continuity of the $\phi_i$, we may choose $\beta>0$ small enough so that $\sum_{i=1}^{n-1}\Norm{\phi_i(x)- \phi_i(y)}<\frac \alpha2$ whenever $\norm{x-y}<\beta$.  Thus, if $\norm{x-y}<\min\{\beta, t_{n-1}\}$, we have
\[\begin{split}
\norm{\phi(x)-\phi(y)}
&\leqslant \sum_{i=1}^{n-1}\Norm{\phi_i(x)- \phi_i(y)}+ \sum_{i=n}^\infty\Norm{\phi_i(x)- \phi_i(y)}\\
&<\frac \alpha2+\sum_{i=n}^\infty2^{-i}\\
&<\alpha,
\end{split}\]
showing that $\phi$ is uniformly continuous. A similar argument works for $\psi$.

Now, suppose $\norm {x-y}\geqslant r_m$. Then either $r_n\leqslant \norm {x-y}\leqslant t_n$ or $t_n\leqslant \norm {x-y}\leqslant r_{n+1}$ for some $n\geqslant m$. In the first case,
\[\begin{split}
\Norm{\phi(x)-\phi(y)}
&=\Norm{\sum_{i=1}^\infty\phi_i(x)- \sum_{i=1}^\infty\phi_i(y)}\\
&\geqslant \Norm{\sum_{i=1}^n\phi_i(x)- \sum_{i=1}^n\phi_i(y)}-\sum_{i=n+1}^\infty\Norm{\phi_i(x)- \phi_i(y)}\\
&\geqslant 2^n-1,
\end{split}\]
while, in the second case,
\[\begin{split}
\Norm{\psi(x)-\psi(y)}
&=\Norm{\sum_{i=1}^\infty\psi_i(x)- \sum_{i=1}^\infty\psi_i(y)}\\
&\geqslant \Norm{\sum_{i=1}^n\psi_i(x)- \sum_{i=1}^n\psi_i(y)}-\sum_{i=n+1}^\infty\Norm{\psi_i(x)- \psi_i(y)}\\
&\geqslant 2^n-1.
\end{split}\]
Thus, 
$$
\norm {x-y}\geqslant r_m\;\;\saa\;\; \norm{\phi(x)-\phi(y)}+\norm{\psi(x)-\psi(y)}\geqslant 2^m-1,
$$
showing that $\phi\oplus\psi\colon X\til E\oplus F$ is expanding. As each of $\phi$ and $\psi$ is uniformly continuous, so is $\phi\oplus\psi$ and therefore also bornologous. It follows that $\phi\oplus\psi$ is a coarse embedding of $X$ into $E\oplus F$.

Let us now return to the construction of $\phi_i, \psi_i, r_i$ and $t_i$.  We begin by letting $t_0=0$, $r_1=\Delta$ and $\phi_1=\frac{2}\delta\sigma$. Now suppose that $\phi_1,\ldots, \phi_n$, $\psi_1,\ldots, \psi_{n-1}$ and $t_0<r_1<t_1<\ldots< r_{n}$ have been defined satisfying the required conditions. As the $\psi_i$ are bounded, let 
$$
S=\sup_{x,y\in X}\Norm{\sum_{i=1}^{n-1}\psi_i(x)-\sum_{i=1}^{n-1}\psi_i(y)}.
$$
Also, as $\omega$ is uniformly continuous, pick $0<\eps<1$ so that $\norm{\omega(x)-\omega(y)}<\frac\delta{(S+2^{n})2^{n}}$ whenever $\norm{x-y}\leqslant \eps$ and let $\psi_n(x)=\frac{S+2^{n}}\delta\omega(\frac \eps{r_n}x)$. Then, if $\norm{x-y}\leqslant r_n$, we have $\norm{\frac \eps{r_n}x-\frac \eps{r_n}y}\leqslant \eps$ and so
$$
\norm{\psi_n(x)-\psi_n(y)}=\frac{S+2^{n}}\delta\NORM{\omega\big(\frac \eps{r_n}x\big)-\omega\big(\frac \eps{r_n}y\big)}<2^{-n}.
$$
On the other hand, if we let $t_n=\frac{r_n\Delta}\eps$, then
$$
\norm{x-y}\geqslant t_n\;\saa\; \norm{\psi_n(x)-\psi_n(y)}\geqslant S+2^n\;\saa\; \Norm{\sum_{i=1}^n\psi_i(x)- \sum_{i=1}^n\psi_i(y)}\geqslant 2^n.
$$
Note that, as $\eps<1$ and $\Delta\geqslant 2$, we have $t_n>2r_n$. A similar construction allows us to find $\phi_{n+1}$ and $r_{n+1}>2t_n$ given $\phi_1,\ldots, \phi_n$, $\psi_1,\ldots, \psi_n$ and $t_0<r_1<\ldots< t_n$.
\end{proof}

In the light of the previous results, it would be very interesting to determine when a coarse embedding can be replaced by a uniformly continuous coarse embedding. As mentioned earlier, Naor \cite{naor-nets} was able to construct a bornologous map between two separable Banach spaces, which is not close to any uniformly continuous map. 
\begin{quest}
Suppose $X$ is a separable Banach space coarsely embedding into a separable Banach space $E$. Is there a uniformly continuous coarse embeddding of $X$ into $E$?
\end{quest}


\section{Cocycles and affine isometric representations}\label{affine actions}
By the Mazur--Ulam Theorem, every surjective isometry $A$ of a Banach space $X$ is {\em affine}, that is, there are a unique invertible linear isometry $T\colon X\til X$ and a vector $\eta\in X$ so that $A$ is given by $A(\xi)=T(\xi)+\eta$ for all $\xi \in X$. 
It follows that, if $\alpha\colon G\curvearrowright X$ is an isometric action of a group $G$ on a Banach space $X$, there is an isometric  linear representation $\pi\colon G\curvearrowright X$, called the {\em linear part} of $\alpha$, and a corresponding {\em cocycle} $b\colon G\til X$ so that
$$
\alpha(g)\xi=\pi(g)\xi +b(g)
$$
for all $g\in G$ and $\xi \in X$. In particular, $b$ is simply the orbit map $g\mapsto \alpha(g)0$. Moreover, the cocycle $b$ then satisfies the {\em cocycle equation}
$$
b(gf)=\pi(g)b(f)+b(g)
$$
for $g,f\in G$. 
Finally, as $\alpha$ is an action by isometries, we have
\[\begin{split}
\norm{ b(f)-b(g)}&=\norm{\alpha(f)0-\alpha(g)0}\\
&=\norm{\alpha(g\inv f)0-0}\\
&=\norm{b(g\inv f)}.
\end{split}\]

Now, if $G$ is a topological group, the action $\alpha$ is continuous, i.e., continuous as a map $\alpha\colon G\times X\til X$, if and only if the linear part $\pi$ is {\em strongly continuous}, that is, $g\in G\mapsto \pi(g)\xi\in X$ is continuous for every $\xi \in X$, and $b\colon G\til X$ is continuous. Moreover, since $b$ is simply the orbit map $g\mapsto \alpha(g)0$, in this case, the cocycle $b\colon G\til X$ is both uniformly continuous and bornologous. We say that the action $\alpha$ is {\em coarsely proper} if and only if  $b\colon G\til X$ is a coarse embedding.

If $\pi\colon G\curvearrowright X$ is a strongly continuous isometric linear representation, we let $Z^1(G,\pi)$ denote the vector space of continuous cocycles $b\colon G\til X$ assiciated to $\pi$. Also, let $B^1(G,\pi)$ denote the linear subspace of {\em coboundaries}, i.e., cocycles  $b$ of the form $b(g)=\xi-\pi(g)\xi$ for some $\xi\in X$. Note that the cocycle $b$ has this form if and only if $\xi$ is fixed by the corresponding affine isometric action $\alpha$ induced by $\pi$ and $b$.

As noted above, continuous cocycles are actually uniformly continuous. But, in the case of cocycles between Banach spaces, we have stronger information available. 

\begin{prop}\label{affine banach}
Let $b\colon X\til E$ be a continuous uncollapsed cocycle between Banach spaces $X$ and $E$,  i.e., there are $\Delta,\delta>0$ so that
$$
\norm x>\Delta\;\saa\; \norm{b(x)}>\delta.
$$
Then $b\colon X\til E$ is a uniform embedding. In fact, there are constants $c,C>0$ so that
$$
c\cdot \min\{\norm{x-y},1\}\leqslant  \norm{b(x)-b(y)}\leqslant C\norm{x-y}+C.
$$
\end{prop}

\begin{proof}
Suppose that $\pi\colon X\curvearrowright E$ is the strongly continuous  isometric linear action for which  $b\colon X\til E$ an uncollapsed continuous cocycle.
As noted above, $b$ is uniformly continous and hence, by the Corson--Klee lemma (Proposition 1.11 \cite{lindenstrauss}), also Lipschitz for large distances. This shows the second inequality.

We now show the first inequality with $c=\min\{\delta, \frac \delta{2\Delta}\}$, which implies uniform continuity of  $b\inv$. Since $\norm{b(x)-b(y)}=\norm{b(x-y)}$, it suffices to verify that 
$$
c\cdot \min\{\norm{x},1\}\leqslant  \norm{b(x)}
$$ 
for all $x\in X$. For $\norm{x}>\Delta$, this follows from our assumption and choice of $c$. So suppose instead that $x\in X\setminus \{0\}$ with $\norm x\leqslant \Delta$ and let $n$ be minimal  so that $n\norm{x}> \Delta$. Then $\norm x\leqslant \frac {\Delta}{n-1}\leqslant \frac {2\Delta}{n}$ and
\[\begin{split}
\delta&< \norm{b(n\cdot x)}\\
&=\norm{\pi^{n-1}(x)b(x)+\pi^{n-2}(x)b(x)+\ldots +b(x)}\\
&\leqslant\norm{\pi^{n-1}(x)b(x)}+\norm{\pi^{n-2}(x)b(x)}+\ldots +\norm{b(x)}\\
&=n\cdot\norm{b(x)},
\end{split}\]
i.e., 
$$
\norm{b(x)}\geqslant \frac \delta n \geqslant \frac \delta {2\Delta}\cdot \frac {2\Delta}{n}\geqslant \frac \delta{2\Delta}\cdot\norm x\geqslant c\norm x
$$ 
as required.
\end{proof}
Since any continuous cocycle is both bornologous and uniformly continuous, we see that any coarsely proper continuous cocycle $b\colon X\til E$ between Banach spaces is simultaneously a uniform and coarse embedding.

The following variation is also of independent interest.
\begin{prop}\label{affine ball}
Let $b\colon X\til E$ be a continuous cocycle between Banach spaces satisfying
$$
\inf_{\norm{x}=r}\norm{b(x)}>0
$$
for some $r> 0$. Then $B_X$ uniformly embeds into $B_E$.
\end{prop}

\begin{proof}
Set $\delta=\inf_{\norm{x}=r}\norm{b(x)}$ and find, by uniform continuity of $b$, some $m$ so that 
$$
\norm{b(x)}>\frac \delta2
$$
whenever $r\leqslant \norm x\leqslant r+\frac 2m$. Then, if $0<\norm x\leqslant \frac 1m$, fix $n$ minimal so that
$r\leqslant \norm{ nx}\leqslant r+\frac 2m$. With this choice of $n$ we have as in the proof of Proposition \ref{affine banach} that
$$
\norm {b(x)}\geqslant \frac \delta{2n}\geqslant \frac\delta{4r}\norm x,
$$
showing that $b\colon \frac 1mB_X\til E$ is a uniform embedding. The proposition follows by rescaling $b$.
\end{proof}
Though we shall return to the issue later, let us just mention that uniform embeddings between balls of Banach spaces has received substantial attention. For example, Y. Raynaud \cite{raynaud} (see also Section 9.5 \cite{lindenstrauss}) has shown that $B_{c_0}$ does not uniformly embed into a stable metric space, e.g., into $L^p([0,1])$ with $1\leqslant p<\infty$.

Our next result replicates the construction from Section \ref{uniform vs coarse} within the context of cocycles.

\begin{prop}\label{cocycle amplification}
Suppose $b\colon X\til E$ is an uncollapsed continuous cocycle between Banach spaces. Then, for every $1\leqslant p<\infty$, there is a coarsely proper continuous cocycle $\tilde b\colon X\curvearrowright \ell^p(E)$. 
\end{prop}

\begin{proof}
Let $b$ be associated to the strongly continuous isometric linear representation $\pi\colon X\curvearrowright E$. 
As in the proof of Lemma \ref{ell p sum}, there are constants $\eps_n$ and $K_n$ so that 
$$
x\mapsto (K_1\cdot b(\eps_1 x), K_2\cdot b(\eps_2 x),\ldots)
$$
defines a simultaneously uniform and coarse embedding of $X$ into $\ell^p(E)$. We claim that the map $\tilde b\colon X\til \ell^p(E)$ so defined is a cocycle for some strongly continuous isometric linear representation $\tilde\pi\colon X\curvearrowright \ell^p(E)$. Indeed, observe that each $x\mapsto K_n\cdot b(\eps_nx)$ is a cocycle for the isometric linear representation $x\mapsto \pi(\eps_n x)$ on $E$, so $\tilde b$ is a cocycle for the isometric linear representation 
$$
\tilde \pi(x)=\pi(\eps_1 x)\otimes \pi(\eps_2x)\otimes \ldots
$$
of $X$ on $\ell^p(E)$.
\end{proof}

\begin{lemme}\label{cocycle isometry group}
Let $E$ be a separable Banach space and ${\rm Isom}(E)$ the group of linear isometries of $E$ equipped with the strong operator topology. If  $\pi\colon {\rm Isom}(E)\curvearrowright \ell^p(E)$ denotes the diagonal isometric linear representation for $1\leqslant p<\infty$, there is a continuous cocycle $b\colon {\rm Isom}(E)\til \ell^p(E)$ associated to $\pi$ that is a uniform embedding  of ${\rm Isom}(E)$ into $\ell^p(E)$.
\end{lemme}

\begin{proof}
Fix a dense subset $\{\xi_n\}_{n\in \N}$ of the sphere $S_E$ and let 
$$
b(g)=\Big(\frac{\xi_1-g(\xi_1)}{2^1},\frac{\xi_2-g(\xi_2)}{2^2},\frac{\xi_3-g(\xi_3)}{2^3},\ldots\Big).
$$ 
Then $b$ is easily seen to be a cocycle for $\pi$. Moreover, by the definition of the strong operator topology, $b$ is a uniform embedding.
\end{proof}

\begin{prop}\label{cocycle amplification banach} 
Supppose $\pi\colon X\til {\rm Isom}(E)$ is an uncollapsed strongly continuous isometric linear representation of a  Banach space $X$ on a separable Banach space $E$. Then, for every $1\leqslant p<\infty$, $X$ admits coarsely proper continuous cocycle $b\colon X\til \ell^p(E)$. 
\end{prop}

\begin{proof}
Let $c\colon {\rm Isom}(E)\til \ell^p(E)$ be the cocycle given by Lemma \ref{cocycle isometry group} associated to the diagonal isometric linear representation $\rho\colon {\rm Isom}(E)\curvearrowright \ell^p(E)$. It follows that $c\circ \pi\colon X\til \ell^p(E)$ is a continuous uncollapsed cocycle associated to the isometric linear representation $\rho\circ \pi\colon X\curvearrowright \ell^p(E)$. By Proposition \ref{cocycle amplification} and the fact that $\ell^p(\ell^p(E))=\ell^p(E)$, we obtain a coarsely proper continuous cocycle $b\colon X \til \ell^p(E)$.
\end{proof}


Cocycles between Banach spaces are significantly more structured maps than general maps. For example, as shown in Proposition \ref{affine banach}, a coarsely proper continuous cocycle $b\colon X\til E$ between Banach spaces $X$ and $E$ is automatically both a uniform and coarse embedding, but, moreover, $b$ also preserves a certain amount of algebraic structure, depending on the isometric linear representation $\pi\colon X\curvearrowright E$ of which it is a cocycle.

To study these algebraic features, we must introduce a topology on the space of cocycles. So fix a strongly continuous isometric linear representation $\pi\colon G\curvearrowright E$  of a topological group $G$ on a Banach space $E$.
Every compact set $K\subseteq G$ determines a seminorm $\norm\cdot_K$ on $Z^1(G,\pi)$ by $\norm{ b}_K=\sup_{g\in K}\norm{b(g)}$ and the family of seminorms thus obtained endows $Z^1(G,\pi)$ with a locally convex topology. With this topology, one sees that a cocycle $b$ belongs to the closure $\overline{B^1(G,\pi)}$ if and only if the corresponding affine action $\alpha=(\pi,b)$ {\em almost has fixed points}, that is, if for any compact set $K\subseteq G$ and $\eps>0$ there is some $\xi=\xi_{K,\eps}\in E$ verifying
$$
\sup_{x\in K}\norm{\big(\pi(x)\xi+b(x)\big)-\xi}=\sup_{x\in K}\norm{b(x)-\big(\xi-\pi(x)\xi\big)}<\eps.
$$
Elements of $\overline{B^1(G,\pi)}$ are called {\em almost coboundaries}.

Note that, if $b$ is a coboundary, then $b(G)$ is a bounded subset of $E$. Conversely, suppose $b(G)$ is a bounded set and $E$ is reflexive. Then any orbit $\ku O$ of the corresponding affine action is bounded and its closed convex hull $C=\overline{\rm conv}(\ku O)$ is a weakly compact convex set on which $G$ acts by affine isometries. It follows by the Ryll-Nardzewski fixed point theorem \cite{ryll} that $G$ fixes a point on $C$, meaning that $b$ must be a coboundary.

Now, if $b\in\overline{B^1(G,\pi)}$ and, for every compact $K\subseteq G$, we can choose $\xi=\xi_{K,1}$ above to have arbitrarily large norm, we see that the supremum
$$
\sup_{x\in K}\Norm{\pi(x)\frac \xi{\norm \xi}-\frac\xi{\norm \xi}}<\frac{\sup_{x\in K}\norm{b(x)}+1}{\norm \xi}
$$
can be made arbitrarily small, which means that the linear action $\pi$ {\em almost has invariant unit vectors}.
If, on the other hand, for some $K$ the choice of $\xi_{K,1}$ is bounded (but non-empty), then the same bound holds for any compact $K'\supseteq K$, whereby  we find that $b(G)\subseteq E$ is a bounded set and so, assuming $E$ is reflexive, that $b\in B^1(G,\pi)$. This shows that, if $E$ is reflexive and $\pi$ does not almost have invariant unit vectors, then $B^1(G,\pi)$ is closed in $Z^1(G,\pi)$.

We define the {\em first cohomology group} of $G$ with coefficients in $\pi$ to be the quotient space $H^1(G,\pi)=Z^1(G,\pi)/B^1(G,\pi)$, while the {\em reduced cohomology group} is $\overline{H^1}(G,\pi)=Z^1(G,\pi)/\overline{B^1(G,\pi)}$.

If $E$ is separable and reflexive, the Alaoglu--Birkhoff decomposition theorem \cite{alaoglu} implies that $E$ admits $\pi(G)$-invariant decomposition into closed linear subspaces $E=E^G\oplus E_G$, where $E^G$ is the set of $\pi(G)$-fixed vectors. We can therefore write $b=b^G\oplus b_G$, where $b^G\colon X\til E^G$ and $b_G\colon X\til E_G$ are cocycles for $\pi$.
In particular, 
$$
b^G(xy)=\pi(x)b^G(y)+b^G(x)=b^G(y)+b^G(x)=b^G(x)+b^G(y),
$$ 
i.e., $b^G$ is a continuous homomorphism from $G$ to $(E^G,+)$.     Also, Theorem 2 of \cite{BRS}, implies that, if $G$ is abelian, then $\overline{H^1}(G, \pi|_{E_G})=0$ and so $b_G\in \overline{B^1(G,\pi)}$.

Now, suppose $\pi\colon X\curvearrowright E$ is a strongly continuous isometric linear representation of a separable Banach space $X$ on a separable reflexive Banach space $E$. Assume that $b\colon X\til E$ is a  continuous cocycle and let $E=E^X\oplus E_X$ and $b^X\colon X\til E^X$ and $b_X\colon X\til E_X$ be the decompositions as above. Being a continuous additive homomorphism, $b^X$ is a bounded linear operator from $X$ to $E^X$. Also, if $b^X$ is coarsely proper or even just uncollapsed, then $b^X$ must be an isomorphic embedding of $X$ into $E^X$.

Now, since $X$ is abelian, $b_X\colon X\til E_X$  belongs to $\overline{B^1(X,\pi)}$, which means that, for every norm-compact subset $C\subseteq X$ and $\eps>0$, there is $\xi \in E_X$ so that
$$
\norm{\xi-\pi(x)\xi-b_X(x)}<\eps
$$
for all $ x\in C$. 

We summarise the discussion so far in the following lemma.
\begin{lemme}\label{banach on banach}
Suppose $\pi\colon X\curvearrowright E$ is a strongly continuous isometric linear representation of a separable Banach space $X$ on a separable reflexive Banach space $E$ and assume that $b\colon X\til E$ is a  continuous cocycle. Then there is a $\pi(X)$-invariant decomposition $E=E^X\oplus E_X$ and a decomposition $b=b^X\oplus b_X$ so that $b^X\colon X\til E^X$ is a bounded linear operator and $b_X\colon X\til E_X$ belongs to  $\overline{B^1(X,\pi)}$.
\end{lemme}


\section{Amenability}\label{amenability}
Central in our investigation is the concept of amenability, which for general topological groups is defined as follows.
\begin{defi}
A topological group $G$ is {\em amenable} if every continous action $\alpha \colon G\curvearrowright C$ by affine homeohomorphisms  on a compact convex subset $C$ of a locally convex topological vector space $V$ has a fixed point in $C$.
\end{defi} 

If $G$ is a topological group, we let ${\rm LUC}(G)$ be the vector space of bounded left-uniformly continuous functions $\phi\colon G\til \R$ and equip it with the supremum norm induced from the inclusion ${\rm LUC}(G)\subseteq \ell^\infty(G)$. It then follows that the {\em right-regular representation} $\rho\colon G\curvearrowright {\rm LUC}(G)$, $\rho(g)(\phi)=\phi(\cdot\, g)$, is continuous. Moreover, if $G$ is amenable, there exists a $\rho$-invariant mean, that is, a positive continuous linear functional $\go m\colon {\rm LUC}(G)\til \R$ with $\go m(\mathbb 1)=1$, where $\mathbb1 $ is the function with  constant value $1$, and so that $\go m\big(\rho(g)(\phi)\big)=\go m(\phi)$.

While we shall use the existence of invariant means on  ${\rm LUC}(G)$, at one point this does not seem to suffice. Instead, we shall rely on an appropriate generalisation of F\o lner sets present under additional assumptions.

\begin{defi}
A topological group $G$ is said to be {\em approximately compact} if there is a countable chain $K_0\leqslant K_1\leqslant \ldots \leqslant G$ of compact subgroups whose union $\bigcup_nK_n$ is dense in $G$.
\end{defi}

This turns out to be a fairly common phenomenon among non-locally compact Polish groups. For example, the unitary group $U(\ku H)$ of separable infinite-dimensional Hilbert space with the strong operator topology  is approximately compact. Indeed, if $\ku H_1\subseteq \ku H_2\subseteq \ldots\subseteq \ku H$ is an increasing exhaustive sequence of finite-dimensional subspaces and $U(n)$ denotes the group of unitaries pointwise fixing the orthogonal complement $\ku H_n^\perp$, then each $U(n)$ is compact and the union $\bigcup_nU(n)$ is dense in $U(\ku H)$.

More generally, as shown by P. de la Harpe \cite{harpe}, if $M$ is an approximately finite-dimensional von Neumann algebra, i.e., there is an increasing sequence $A_1\subseteq A_2\subseteq\ldots\subseteq M$ of finite-dimensional matrix algebras whose union is dense in $M$ with respect the strong operator topology, then the unitary subgroup $U(M)$ is approximately compact with respect to the strong operator topology. 

Similarly, if $G$ contains a locally finite dense subgroup, this will witness approximate compactness. Again this applies to, e.g., ${\rm Aut}([0,1],\lambda)$ with the weak topology, where the dyadic permutations are dense,  and ${\rm Isom}(\U)$ with the pointwise convergence topology (this even holds for the dense subgroup ${\rm Isom}(\Q\U)$ by an unpublished result of S. Solecki; see \cite{RZ} for a proof).

Of particular interest to us is the case of non-Archimedean Polish groups. By general techniques, these may be represented as automorphism groups of countable locally finite  (i.e., any finitely generated substructure is finite) ultrahomogeneous structures. And, in this setting, we have the following reformulation of approximate compactness.
\begin{prop}[A.S. Kechris \& C. Rosendal \cite{turbulence}]
Let $\bf M$ be a locally finite, countable, ultrahomogeneous structure.
Then ${\rm Aut}(\bf M)$ is approximately compact if and only if, for every finite substructure $\bf A\subseteq \bf M$ and all {\em partial} automorphisms $\phi_1,\ldots,
\phi_n$ of $\bf
A$, there is a  finite substructure  $\bf A\subseteq \bf B\subseteq \bf M$
and {\em full} automorphisms $\psi_1,\ldots,\psi_n$ of $\bf B$ extending $\phi_1,\ldots,\phi_n$ respectively.
\end{prop}

Whereas a locally compact group is amenable if and only if it admits a F\o lner sequence, there is no similar characterisation of general amenable groups. Nevertheless, one may sometimes get by with a little less, which we isolate in the following definition.
\begin{defi}\label{folner}
A Polish group $G$ is said to be {\em F\o lner amenable} if either
\begin{enumerate}
\item $G$ is approximately compact, or
\item there is a continuous homomorphism $\phi\colon H\til G$  from a locally compact second countable amenable group $H$ so that $G=\ov{\phi[H]}$. 
\end{enumerate}
\end{defi}

Apart from the approximately compact or locally compact amenable groups, easy examples of F\o lner amenable Polish groups are, e.g., Banach spaces or, more generally, abelian groups. Indeed, every abelian Polish group $G$ contains a countable dense subgroup $\Gamma$, which, viewed as a discrete group, is amenable and maps densely into $G$.


\section{Embeddability in Hilbert spaces}
We shall now consider Hilbert valued cocycles, for which we need some background material on kernels conditionally of negative type. The well-known construction of inner products presented here originates in work of E. H. Moore \cite{moore}. A full treatment can be found, e.g., in \cite{bekka}, Appendix C.
\begin{defi}
A (real-valued) {\em kernel conditionally of negative type} on a set $X$ is a  function $\Psi\colon X\times X\til \R$  so that
\begin{enumerate}
\item $\Psi(x,x)=0$ and $\Psi(x,y)=\Psi(y,x)$ for all $x,y\in X$,
\item for all $x_1,\ldots, x_n\in X$ and $r_1,\ldots, r_n\in \R$ with $\sum_{i=1}^nr_i=0$, we have
$$
\sum_{i=1}^n\sum_{j=1}^n r_ir_j\Psi(x_i,x_j)\leqslant 0.
$$
\end{enumerate}
\end{defi} 
For example, if $\sigma\colon X\til \ku H$ is any mapping from $X$ into a Hilbert space $\ku H$, then a simple calculation shows that
$$
\sum_{i=1}^n\sum_{j=1}^n r_ir_j\norm{\sigma(x_i)-\sigma(x_j)}^2=-2\Norm{\sum_{i=1}^nr_i\sigma(x_i)}\leqslant 0,
$$
whenever $\sum_{i=1}^nr_i=0$, which implies that $\Psi(x,y)=\norm{\sigma(x)-\sigma(y)}^2$ is a kernel conditionally of negative type.

Suppose that $\Psi$ is a kernel conditionally of negative type on  a set $X$ and let $\M(X)$  denote the vector space of finitely supported real valued functions $\xi$ on $X$ of mean $0$, i.e., $\sum_{x\in X}\xi(x)=0$. We define a positive symmetric linear form  $\langle\cdot\del \cdot\rangle_\Psi$ on $\M(X)$ by
$$
\Big\langle\sum_{i=1}^nr_i\delta_{x_i}\Del \sum_{j=1}^ks_j\delta_{y_i}\Big\rangle_\Psi=-\frac 12\sum_{i=1}^n\sum_{j=1}^k r_is_j\Psi(x_i,y_j).
$$
Also, if $N_\Psi$ denotes the null-space
$$
N_\Psi=\{\xi\in \M(X)\del \langle\xi\del \xi\rangle_\Psi=0\},
$$
then $\langle\cdot\del \cdot\rangle_\Psi$ defines an inner product on the quotient $\M(X)/N_\Psi$ and 
we obtain a real Hilbert space $\ku K$ as the completion of $\M(X)/N_\Psi$ with respect to $\langle\cdot\del \cdot\rangle_\Psi$.

We remark that, if $\Psi$ is defined by a map $\sigma\colon X\til \ku H$ as above and $e\in X$ is any choice of base point, the map $\phi_e\colon X\til \ku K$ defined by $\phi_e(x)=\delta_x-\delta_e$ satisfies $\norm{\phi_e(x)-\phi_e(y)}_\ku K=\norm{\sigma(x)-\sigma(y)}_\ku H$. Indeed,
\[\begin{split}
\norm{\phi_e(x)-\phi_e(y)}^2_\ku K
&=\langle \phi_e(x)-\phi_e(y)\del \phi_e(x)-\phi_e(y)\rangle\\
&=\langle \delta_x-\delta_y\del \delta_x-\delta_y\rangle\\
&=-\frac 12\big(\Psi(x,x)+\Psi(y,y)-\Psi(x,y)-\Psi(y,x)\big)\\
&=\Psi(x,y)\\
&=\norm{\sigma(x)-\sigma(y)}_\ku H^2.
\end{split}\]

Also, if $G\curvearrowright X$ is an action of a group $G$ on $X$ and $\Psi$ is $G$-invariant, i.e., $\Psi(gx,gy)=\Psi(x,y)$, this action lifts to an action  $\pi\colon G\curvearrowright\M(X)$ preserving the form $\langle\cdot\del\cdot\rangle_\Psi$  via $\pi(g)\xi=\xi(g\inv\,\cdot\,)$. It follows that $\pi$ factors through to an orthogonal (i.e., isometric linear) representation $G\curvearrowright\ku K$.

A version of Lemma \ref{pre-maurey} below is originally due to I. Aharoni, B. Maurey and B. S. Mityagin \cite{maurey} for the case of abelian groups and has been extended and refined several times recently in connection with the coarse geometry of Banach spaces and locally compact groups (see, e.g., \cite{tessera, randrianarivony, randrianarivony2}).
Since more care is needed when dealing with general amenable as opposed to locally compact amenable or abelian groups, we include a full proof.

Let us first recall that, if $\sigma \colon X\til Y$ is a map between metric spaces, the exact compression modulus $\tilde \kappa$ of $\sigma$ is given by $\tilde\kappa(t)=\inf_{d(x,x')=t}d(\sigma(x), \sigma(x'))$.

\begin{lemme}\label{pre-maurey}
Suppose $d$ is a compatible left-invariant metric on an amenable  topological group $G$ and $\sigma\colon (G,d)\til \ku H$ is a uniformly continuous and bornologous map into a Hilbert space $\ku H$ with exact compression modulus $\tilde \kappa$ and expansion modulus $\theta$. 
Then there is a continuous $G$-invariant kernel conditionally of negative type $\Psi\colon G\times G\til \R_+$ satisfying
$$
\tilde\kappa\big(d(g,f)\big)^2\leqslant \Psi(g,f)\leqslant \theta\big(d(g,f)\big)^2.
$$
\end{lemme}

\begin{proof} 
For fixed $g,h\in G$, we define a function $\phi_{g,h}\colon G\til \R$ via
$$
\phi_{g,h}(f)=\norm{\sigma(fg)-\sigma(fh)}^2.
$$
Since, for all $g,h,f\in G$, we have 
$$
\tilde\kappa\big(d(g,h)\big)^2=\tilde\kappa\big(d(fg,fh)\big)^2\leqslant \norm{\sigma(fg)-\sigma(fh)}^2\leqslant \theta\big(d(g,h)\big)^2,
$$
it follows that
$$
\tilde\kappa\big(d(g,h)\big)^2\leqslant \phi_{g,h}\leqslant\theta\big(d(g,h)\big)^2
$$
and so, in particular,  $\phi_{g,h}\in \ell^\infty(G)$. 

We claim that $\phi_{g,h}$ is {\em left}-uniformly continuous, i.e., that  for all $\eps>0$ there is $W\ni 1$ open so that $|\phi_{g,h}(f)-\phi_{g,h}(fw)|<\eps$, whenever $f\in G$ and $w\in W$. To see this, take some $\eta>0$ so that $4\eta\norm{\phi_{g,h}}_\infty+4\eta^2<\eps$ and find, by uniform continuity of $\sigma$, some open $V\ni1$ so that   $\norm{\sigma(f)-\sigma(fv)}<\eta$ for all $f\in G$ and $v\in V$. 
Pick also $W\ni 1$ open so that $Wg\subseteq gV$ and $Wh\subseteq hV$. Then, if $f\in G$ and $w\in W$, there are $v_1,v_2\in V$ so that $wg=gv_1$ and $wh=hv_2$, whence
\[\begin{split}
\big|\phi_{g,h}(f)-\phi_{g,h}(fw)\big|
&=\Big|  \norm{\sigma(fg)-\sigma(fh)}^2- \norm{\sigma(fwg)-\sigma(fwh)}^2\Big|  \\
&=\Big|  \norm{\sigma(fg)-\sigma(fh)}^2- \norm{\sigma(fgv_1)-\sigma(fhv_2)}^2\Big|\\ 
&<4\eta\norm{\phi_{g,h}}_\infty+4\eta^2\\
&<\eps.
\end{split}\]
Thus, every $\phi_{g,h}$ belongs to the closed linear subspace ${\rm LUC}(G)\subseteq \ell^\infty(G)$ of left-uniformly continuous bounded real-valued functions on $G$ and a similar calculation shows that the map $(g,h)\in G\times G\mapsto \phi_{g,h}\in \ell^\infty(G)$ is continuous. 

Now, since $G$ is amenable, there exists a mean $\go m$ on ${\rm LUC}(G)$ invariant under the {\em right}-regular representation $\rho\colon G\curvearrowright {\rm LUC}(G)$ given by $\rho(g)\big(\phi\big)=\phi(\,\cdot\, g)$. Using this, we can define a continuous kernel $\Psi\colon G\times G\til \R$ by
$$
\Psi(g,h)=\go m(\phi_{g,h})
$$
and  note that $\Psi(fg,fh)=\go m(\phi_{fg,fh})=\go m\big(\rho(f)\big(\phi_{g,h}\big)\big)=\go m(\phi_{g,h})=\Psi(g,h)$ for all $g,h,f\in G$. 

We claim that $\Psi$ is a kernel conditionally of negative type. To verify this, let $g_1,\ldots, g_n\in G$ and $r_1,\ldots, r_n\in \R$ with $\sum_{i=1}^nr_i=0$. Then, for all $f\in G$, 
$$
\sum_{i=1}^n\sum_{j=1}^n r_ir_j\phi_{g_i,g_j}(f)=\sum_{i=1}^n\sum_{j=1}^n r_ir_j\norm{\sigma(fg_i)-\sigma(fg_j)}^2\leqslant 0,
$$
since $(g,h)\mapsto \norm{\sigma(fg)-\sigma(fh)}^2$ is a kernel conditionally of negative type. As $\go m$ is positive, it follows that also
$$
\sum_{i=1}^n\sum_{j=1}^n r_ir_j\Psi(g_i,g_j)
= \go m\Big(\sum_{i=1}^n\sum_{j=1}^n r_ir_j\phi_{g_i,g_j}\Big)
\leqslant 0.
$$

Finally, as $\go m$ is a mean and 
$$
\tilde\kappa\big(d(g,h)\big)^2\leqslant \phi_{g,h}\leqslant \theta\big(d(g,h)\big)^2,
$$
it follows that
$$
\tilde\kappa\big(d(g,h)\big)^2\leqslant \Psi(g,h)\leqslant \theta\big(d(g,h)\big)^2
$$
as required.
\end{proof}

\begin{thm}\label{maurey}
Suppose $d$ is a compatible left-invariant metric on an amenable  topological group $G$ and $\sigma\colon (G,d)\til \ku H$ is a uniformly continuous and bornologous map into a Hilbert space $\ku H$ with exact compression modulus $\tilde \kappa$ and expansion modulus $\theta$. 

Then there is a continuous  cocycle  into a real Hilbert space  $b\colon G\til \ku K$ so that 
$$
\tilde \kappa\big(d(g,f)\big)\leqslant \norm{b(g)-b(f)}\leqslant \theta\big(d(g,f)\big),
$$
for all $g,f\in G$.
\end{thm}

\begin{proof}
Let $\Psi$ be the $G$-invariant kernel conditionally of negative type given by Lemma \ref{pre-maurey}.
As above, we define a positive symmetric form $\langle\cdot\del \cdot\rangle_\Psi$ on $\M(G)$.
Note that, since $\Psi$ is $G$-invariant, the form $\langle\cdot\del \cdot\rangle_\Psi$ is invariant under the left-regular representation $\lambda\colon G\curvearrowright \M(G)$ given by $\lambda(g)(\xi)=\xi(g\inv\,\cdot\,)$ and so $\lambda$ induces a strongly continuous  orthogonal representation $\pi$ of $G$ on the Hilbert space completion $\ku K$ of $\M(G)/N_\Psi$.

Moreover, as is easily checked, the map $b\colon G\til \ku K$ given by $b(g)=(\delta_g-\delta_1)+N_\Psi$ is a cocycle for $\pi$. Now
$$
\norm{b(g)}^2=\langle \delta_g-\delta_1 \del \delta_g-\delta_1\rangle_\Psi=\Psi(g,1),
$$
whence
$$
\tilde \kappa\big(d(g,1)\big)\leqslant \norm{b(g)}\leqslant \theta\big(d(g,1)\big).
$$
Replacing $g$ with $f\inv g$, the theorem follows.
\end{proof}

The first application concerns uniform embeddability of balls. Recall, for example, that a Banach space $X$ whose unit ball $B_X$ is uniformly embeddable into a Hilbert space must have finite cotype (Proposition 5.3 \cite{raynaud}). 
\begin{cor}\label{balls hilbert}
Let $\sigma\colon X\til \ku H$ be a uniformly continuous map from a Banach space $X$ into a Hilbert space satisfying
$$
\inf_{\norm{x-y}=r}\norm{\sigma(x)-\sigma(y)}>0
$$
for just some $r>0$. Then $B_X$ is uniformly embeddable into Hilbert space. 
\end{cor}

\begin{proof}
Observe that the exact compression modulus of $\sigma$ satisfies $\tilde\kappa_\sigma(r)>0$. Thus, 
applying Theorem \ref{maurey}, we obtain a Hilbert valued continuous cocycle $b\colon X\til \ku K$ with $\tilde\kappa_b(r)>0$. Finally, an application of Proposition \ref{affine ball}, shows that $B_X$ uniformly embeds into $B_\ku K$.
\end{proof}

The requirement that $\sigma$ be uniformly continuous in Theorem \ref{maurey} above is somewhat superfluous. Indeed, W. B. Johnson and N. L. Randrianarivony (Step 0 in \cite{randrianarivony}) have shown that any separable Banach space admitting a coarse embedding into a Hilbert space also has a uniformly continuous coarse embedding into a Hilbert space and the proof carries over directly to prove the following.
\begin{lemme}\cite{randrianarivony}\label{randrianarivony}
Suppose $\sigma\colon (X,d) \til \ku H$ is a map from a metric space into a Hilbert space and assume that $\sigma$ is Lipschitz for large distances and expanding. Then $(X,d)$ admits  a uniformly continuous coarse embedding into a Hilbert space.
\end{lemme}

For a topological group, the question is then when may may work with Lipshcitz for large distance maps rather than bornologous maps.
\begin{lemme}\label{lipschitz}
Let $\sigma\colon G\til (X,d)$ be a bornologous map from a Polish group into a metric space. Then there is a compatible left-invariant metric $\partial$ on $G$ so that 
$$
\sigma\colon (G,\partial)\til (X,d)
$$
is Lipschitz for large distances.
\end{lemme}

\begin{proof}
For $a\in G$, set 
$$
w(a)=\sup_{x\inv y=a}d(\sigma(x),\sigma(y)) 
$$
and observe that, as $\sigma$ is bornologous, $w(a)=w(a\inv)<\infty$.
We claim that there is some symmetric open identity neighbourhood $V$ so that 
$$
C=\sup_{a\in V}w(a)<\infty.
$$
If not, there are $a_n\til 1$ so that $w(a_n)>n$. But then $K=\{a_n,1\}_n$ is a compact and thus relatively (OB) set and 
$$
E_K=\{(x,y)\in G\times G\del x\inv y\in K\}
$$
a coarse entourage on $G$. As $\sigma$ is bornologous, we see that $(\sigma\times \sigma)E_K$ is a coarse entourage in $X$, i.e., $\sup_{x\inv y \in K}d(\sigma(x),\sigma(y))<\infty$, contradicting the assumptions on $K$.

So pick $V$ as claimed and let $D\leqslant 1$ be some compatible left-invariant metric on $G$. Fix also $\eps>0$ so that $V$ contains the ball $B_D(2\eps)$. Define 
\[\begin{split}
\partial(x,y)=\inf\Big(\sum_{i\in B}D(a_i,1)+\sum_{i\notin B}\big(w(a_i)+1\big)\Del x=ya_1\cdots a_n\;\;\&\;\; 
a_i\in V\text{ for }i\in B\Big).
\end{split}\]
Since $D$ is a compatible metric and $V\ni 1$ is open, $\partial$ is a continuous left-invariant \'ecart. Moreover, as $\partial\geqslant D$, we see that $\partial$ is a compatible metric on $G$. 

Now, suppose $x,y\in G$ are given and write $x=ya_1\cdots a_n$ for some  $a_i\in G$ with $a_i\in  V$ for $i\in B$ so that $\sum_{i\in B}D(a_i,1)+\sum_{i\notin B}\big(w(a_i)+1\big)\leqslant \partial(x,y)+\eps$. Observe that, if $D(a_i,1)<\eps$ and $D(a_{i+1}, 1)< \eps$ for some $i<n$, then $a_ia_{i+1}\in V$,
so, by coalescing $a_i$ and $a_{i+1}$ into a single term $a_ia_{i+1}\in V$, we only decrease the final sum. We may therefore assume that, for every $i<n$, either $D(a_i,1)\geqslant \eps$ or $D(a_{i+1}, 1)\geqslant \eps$. It thus follows that 
$$
 \partial(x,y)+\eps\geqslant \sum_{i\in B}D(a_i,1)+\sum_{i\notin B}\big(w(a_i)+1\big)\geqslant \sum_{i\leqslant n}D(a_i,1)\geqslant \frac{n-1}2\cdot \eps,
$$
i.e., $n\leqslant \frac {2\partial(x,y)}\eps+3$.

Therefore,
\[\begin{split}
d(\sigma(x),\sigma(y))
\leqslant&
 \sum_{i\in B}d\big(\sigma(ya_1\cdots a_{i}), \sigma(ya_1\cdots a_{i-1})\big)\\
 &+\sum_{i\notin B}d\big(\sigma(ya_1\cdots a_{i}), \sigma(ya_1\cdots a_{i-1})\big)\\
 \leqslant &
C\cdot n+\sum_{i\notin B}w(a_{i})\\
\leqslant &
\Big(\frac{2C}\eps+1\Big)\cdot\partial(x,y) +3C+\eps.
\end{split}\]
In other words, $\sigma\colon (G,\partial)\til (X,d)$ is Lipschitz for large distances.
\end{proof}

We can now extend the definition of the Haage\-rup property (see, e.g., Definition 2.7.5 \cite{bekka} or \cite{CCJJV}) from locally compact groups to the category of all topological groups.
\begin{defi}
A topological group $G$ is said to have the {\em Haagerup property} if it admits a coarsely proper continuous affine isometric action on a Hilbert space.
\end{defi}

Thus, based on Theorem \ref{maurey} and Lemmas \ref{randrianarivony} and \ref{lipschitz}, we have the following reformulation of the Haage\-rup property for amenable Polish groups.

\begin{thm}\label{haagerup equiv}
The following conditions are equivalent for an amenable Polish group $G$,
\begin{enumerate}
\item $G$ coarsely embeds into a Hilbert space,
\item $G$ has the Haagerup property.
\end{enumerate}
\end{thm}

\begin{proof}(2)$\saa$(1): If $\alpha\colon G\curvearrowright \ku H$ is a coarsely proper continuous affine isometric action, with corresponding cocycle $b\colon G\til \ku H$, then $b\colon G\til \ku H$ is a uniformly continuous coarse embedding.

(1)$\saa$(2): If $\eta\colon G\til \ku H$ is a coarse embedding, then by Lemma \ref{lipschitz} there is a compatible left-invariant metric $d$ on $G$ so that $\eta\colon (G,d)\til \ku H$ is Lipschitz for large distances. Since $\eta$ is a coarse embedding, it must be expanding with respect to the metric $d$.
It follows from Lemma \ref{randrianarivony} that $\eta\colon (G,d)\til \ku H$ may also be assumed to be uniformly continuous. Thus, by Theorem \ref{maurey}, there is a continuous affine isometric action $\alpha\colon G\curvearrowright \ku K$ on a Hilbert space $\ku K$ with associated cocycle $b\colon G\til \ku K$ so that, for all $g\in G$, 
$$
\kappa_\eta\big(d(g,f)\big)\leqslant \norm{b(g)-b(f)}\leqslant\theta_\eta\big(d(g,f)\big).
$$
Since $\eta$ is a coarse embedding, the cocycle $b\colon G\til \ku K$ is coarsely proper and so is the action $\alpha$. 
\end{proof}

U. Haagerup \cite{haagerup} initially showed that finitely generated free groups have the Haagerup property. It is also known that amenable locally compact groups  \cite{BCV} (see also  \cite{CCJJV}) have the Haagerup property. However, this is not the case for amenable Polish groups. For example, a separable Banach space not coarsely embedding into Hilbert space such as $c_0$ of course also fails the Haagerup property.

There is also a converse to this. Namely, E. Guentner and J. Kaminker \cite{GK} showed that, if a finitely generated discrete group $G$ admits a affine isometric action on a Hilbert space whose cocycle $b$ grows faster than the square root of the word length, then $G$ is amenable (see \cite{tessera} for the generalisation to the locally compact case). It is not clear what, if any, generalisation of this is possible to the setting of (OB) generated Polish groups. For example, by Theorem 7.2 \cite{OB}, the homeomorphism group of the $n$-sphere $\mathbb S^n$ has property (OB) and thus is quasi-isometric to a point. It therefore trivially fulfills the assumptions of the Guentner--Kaminker theorem, but is not amenable.

It is by now well-known that the are Polish groups admitting no non-trivial continuous unitary or, equivalently, orthogonal, representations (the first example seems to be due to J. P. R. Christensen and W. Herer \cite{christensen}). Also many of these examples are amenable. While the non-existence of unitary representations may look like a local condition on the group that can be detected in neighbourhoods of the identity, the following result shows that under extra assumptions this is also reflected in the large scale behaviour of the group.

\begin{prop}
Suppose $G$ is an amenable Polish group with no non-trivial unitary representations. Then $G$ is either coarsely equivalent to a point or is not coarsely embeddable into Hilbert space.
\end{prop}

\begin{proof}Suppose $G$ is not coarsely equivalent to a point.
Then, by Theorem \ref{haagerup equiv}, if $G$ is coarsely embeddable into Hilbert space, there is a coarsely proper continuous affine isometric action of $G$ on Hilbert space. So either the linear part $\pi$ is a non-trivial orthogonal representation of $G$ or the corresponding cocycle $b\colon G\til \ku H$ is a coarsely proper continuous homomorphism. Thus, as $G$ is not coarsely equivalent to a point,  in the second case, $b$ is unbounded and so composing with an appropriate linear functional, we get an unbounded homomorphism into $\R$. In either case, $G$ will admit a non-trivial orthogonal and thus also unitary representation contrary to our assumptions.
\end{proof}

Finally, let us sum up the equivalences for Banach spaces.
\begin{thm}\label{haagerup banach}
The following conditions are equivalent for a separable Banach space $X$,
\begin{enumerate}
\item $X$ coarsely embeds into a Hilbert space,
\item $X$ uniformly embeds into a Hilbert space,
\item $X$ admits an uncollapsed uniformly continuous map into a Hilbert space,
\item $X$ has the Haagerup property.
\end{enumerate}
\end{thm}

Here Condition (4) of Theorem \ref{haagerup banach} is a priori the strongest of the four since a continuous coarsely proper cocycle $b\colon X\til \ku H$ will be simultaneously a uniform and coarse embedding. The equivalence of (1) and (2) was proved in \cite{randrianarivony} (the implication from (3) to (1) and (2) of course also being a direct consequence of Theorem \ref{ell p sum}), while the implication from (2) to (4) following from Theorem \ref{haagerup equiv}. A seminal characterisation of uniform embeddability into Hilbert spaces utilising a version of Lemma \ref{pre-maurey} appears in \cite{maurey}.


\section{Preservation of local structure}\label{super-refl}
Weakening the geometric restrictions on the phase space from euclidean to uniformly convex, we still have a result similar to Theorem \ref{maurey}. However, in this case, we do not know if amenability suffices, but must rely on strengthenings of this.

Before we state the next result, we recall that a Banach space $X$ is said to be {\em finitely representable} in a Banach space $Y$ if, for every finite-dimensional subspace $F\subseteq X$ and every $\eps>0$, there is a linear embedding $T\colon F\til Y$ so that $\norm T\cdot \norm{T\inv}<1+\eps$. We also say that $X$ is {\em crudely finitely representable} in $Y$ if there is a constant $K$ so that every finite-dimensional subspace of $X$ is $K$-isomorphic to a subspace of $Y$.
Finally, if $E$ is a Banach space and $1\leqslant p<\infty$, we let $L^p(E)$ denote the Banach space of equivalence classes of measurable functions $f\colon [0,1]\til E$ so that the $p$-norm
$$
\norm{f}_{L^p(E)}=\Big(\int_0^1 \norm{f}_E^p \;d\lambda\Big)^{1/p}
$$
is finite.

The next theorem is the basic result for preservation of local structure. Earlier versions of this are due to A. Naor and Y. Peres \cite{naor-peres} and V. Pestov \cite{pestov}  respectively for finitely generated amenable and locally finite discrete groups.  

\begin{thm}\label{pestov}
Suppose that $d$ is a continuous left-invariant \'ecart on a F\o lner amenable Polish group $G$ and  $\sigma\colon (G,d)\til E$ is a uniformly continuous and bornologous map into a Banach space $E$ with exact compression modulus $\tilde\kappa$ and expansion modulus $\theta$. 

Then, for every $1\leqslant p<\infty$, there is a continuous affine isometric action of $G$ on a Banach space $V$ finitely representable in $L^p(E)$ with corresponding cocycle $b$ satisfying
$$
\tilde\kappa\big(d(x,y)\big)\leqslant \norm{b(x)-b(y)}_V\leqslant \theta\big(d(x,y)\big)
$$
for all $x,y\in G$.
\end{thm}

\begin{proof}
Let us first assume that $G$ satisfies the second assumption of Definition \ref{folner} and let $\phi \colon H\til G$ be the corresponding mapping. Then, replacing $H$ by the amenable group $H/\ker \phi$, we may suppose that $\phi$ is injective. Let also $|\cdot|$ denote a right-invariant Haar measure on $H$. Since $H$ is amenable and locally compact, there is  a F\o lner sequence $\{F_n\}$, i.e., a sequence of Borel subsets $F_n\subseteq H$ of finite positive meaure so that $\lim_n\frac{|F_nx\triangle F_n|}{|F_n|}=0$ for all $x\in H$. Let also $L^p(F_n,E)$ denote the Banach space of $p$-integrable functions $f\colon F_n\til E$ with norm 
$$
\norm{f}_{L^p(F_n,E)}=\Big(\int_{F_n}\norm{f(x)}_E^p\Big)^{1/p}.
$$

Now fix a non-principal ultrafilter $\ku U$ on $\N$ and let $\big(\prod_nL^p(F_n,E)\big)_\ku U$ denote the corresponding ultraproduct. That is, if we equip
$$
W=\{(\xi_n)\in \prod_nL^p(F_n,E)\del \sup_n\norm{\xi_n}_{L^p(F_n,E)}<\infty\}
$$ 
with the semi-norm
$$
\norm{(\xi_n)}_\ku U=\lim_\ku U\norm{\xi_n}_{L^p(F_n,E)}
$$
and let $N=\{(\xi_n)\in W\del \norm{(\xi_n)}_W=0\}$ denote the corresponding null-space, then $\big(\prod_n\ell^p(F_n,E)\big)_\ku U$ is the quotient  $W/N$. For simplicity of notation, if $(\xi_n)\in W$, we denote the element $(\xi_n)+N\in \big(\prod_nL^p(F_n,E)\big)_\ku U$ by $(\xi_n)_\ku U$.

Consider the  linear operator $\Theta\colon L^\infty(H, E)\til \big(\prod_nL^p(F_n,E)\big)_\ku U$ given by 
$$
\Theta(f)= \big(|F_n|^{-1/p}\cdot f\begr_{F_n}\big)_\ku U.
$$
Then $\norm{f}_{\ku U,p}=\Norm{\Theta(f)}_\ku U$ defines a semi-norm on $L^\infty(H, E)$ satisfying
\[\begin{split}
\norm{f}_{\ku U, p}^p
&=\NORM{\big(|F_n|^{-1/p}\cdot f\begr_{F_n}\big)_\ku U}_\ku U^p\\
&=\Big(\lim_\ku U\Norm{|F_n|^{-1/p}\cdot f\begr_{F_n}}_{L^p(F_n,E)}\Big)^p\\
&=\lim_\ku U\int_{F_n}\Norm{|F_n|^{-1/p}\cdot f(x)}_{E}^p\\
&=\lim_\ku U\frac 1{|F_n|}\int_{F_n}\Norm{f(x)}_E^p.
\end{split}\]

We claim that $\norm{\cdot }_{\ku U,p}$ is invariant under the right-regular representation $\rho \colon H\curvearrowright L^\infty(H,E)$.
Indeed, for all $f\in L^\infty(H,E)$ and $y\in H$, 
\[\begin{split}
\norm{\rho(y)f}_{\ku U, p}^p-\norm{f}_{\ku U, p}^p
&=\lim_\ku U\frac 1{|F_n|}\int_{x\in F_n}\Norm{\big(\rho(y)f\big)(x)}_E^p-\lim_\ku U\frac 1{|F_n|}\int_{x\in F_n}\Norm{f(x)}_E^p\\
&=\lim_\ku U\frac 1{|F_n|}\Big(\int_{x\in F_n}\Norm{f(xy)}_E^p-\int_{x\in F_n}\Norm{f(x)}_E^p\Big)\\
&=\lim_\ku U\frac 1{|F_n|}\Big(\int_{x\in F_ny}\Norm{f(x)}_E^p-\int_{x\in F_n}\Norm{f(x)}_E^p\Big)\\
&\leqslant \lim_\ku U\frac 1{|F_n|}\int_{x\in F_ny\triangle F_n}\Norm{f(x)}_E^p\\
&\leqslant \norm{f}_{L^\infty(H, E)}^p\cdot \lim_\ku U \frac {|F_ny\triangle F_n|}{|F_n|}\\
&=0.
\end{split}\]

Since $\phi\colon H\til G$ is a homomorphism and $d$ is left-invariant, for $x,y,z\in H$ we have
$$
d\big(\phi(x),\phi(y)\big)=d\big(\phi(z)\phi(x),\phi(z)\phi(y)\big)=d\big(\phi(zx),\phi(zy)\big)
$$
and hence
\[\begin{split}
\tilde\kappa\big[d\big(\phi(x),\phi(y)\big)\big]
&=\tilde\kappa\big[d\big(\phi(zx),\phi(zy)\big)\big]\\
&\leqslant \Norm{\sigma\phi(zx)-\sigma\phi(zy)}_E\\
&= \Norm{\big(\rho(x)\sigma\phi\big)(z)-\big(\rho(y)\sigma\phi\big)(z)}_E\\
&\leqslant 
\theta\big[d\big(\phi(x),\phi(y)\big)\big].
\end{split}\]
This shows that, for $x,y\in H$, we have $\big(\rho(x)\sigma\phi\big)-\big(\rho(y)\sigma\phi\big)\in L^\infty(H,E)$ and 
$$
\tilde\kappa\big[d\big(\phi(x),\phi(y)\big)\big]^p
\leqslant 
\frac 1{|F_n|}\int_{F_n}\Norm{\big(\rho(x)\sigma\phi\big)-\big(\rho(y)\sigma\phi\big)}_E^p
\leqslant 
\theta[d\big(\phi(x),\phi(y)\big)\big]^p
$$
for all $n$. By the expression for $\norm\cdot_{\ku U,p}$, it follows that  
$$
\tilde\kappa\big[d\big(\phi(x),\phi(y)\big)\big]
\leqslant 
\Norm{
\big(\rho(x)\sigma\phi\big)-\big(\rho(y)\sigma\phi\big)
}_{\ku U, p}\leqslant 
\theta\big[d\big(\phi(x),\phi(y)\big)\big].
$$

Therefore, by setting $y=1$, we see that the mapping $b\colon H\til L^\infty(H,E)$ given by
$$
b(x)=\big(\rho(x)\sigma\phi\big)-\sigma\phi
$$
is well-defined. Also, as $b(x)-b(y)= \big(\rho(x)\sigma\phi\big)-\big(\rho(y)\sigma\phi\big)$, we see that
$$
\tilde\kappa\big[d\big(\phi(x),\phi(y)\big)\big]
\leqslant 
\Norm{b(x)-b(y)}_{\ku U,p}\leqslant 
\theta\big[d\big(\phi(x),\phi(y)\big)\big].
$$
Since $\sigma$  uniformly continuous and thus $\lim_{\eps\til 0_+}\theta(\eps)=0$, it follows that
$b$ is uniformly continuous with respect to the metric $d\big(\phi(\cdot),\phi(\cdot)\big)$ on $H$ and the semi-norm  $\norm\cdot_{\ku U,p}$ on $L^\infty(H,E)$. 

Let now $M\subseteq L^\infty(H,E)$ denote the null space of the semi-norm $\norm\cdot_{\ku U,p}$ and $X$ be the completion of $L^\infty(H,E)/M$ with respect to  $\norm\cdot_{\ku U,p}$. Clearly, $X$ is isometrically embeddable into $\big(\prod_nL^p(F_n,E)\big)_\ku U$ and the right-regular representation $\rho\colon H\curvearrowright L^\infty(H,E)$ induces a linear isometric representation of $H$ on $X$, which we continue denoting $\rho$. Similarly, we view $b$ as a map into $X$.

As is easy to see,  $b\in Z^1(H,\rho)$, that is, $b$ satisfies the cocycle identity $b(xy)=\rho(x)b(y)+b(x)$ for $x,y\in H$. So, in particular, the linear span of $b[H]$ is $\rho[H]$-invariant. Moreover, since each $\rho(x)\in \rho[H]$ is an isometry, the same holds for the closed linear span $V\subseteq X$ of $b[H]$. 

We claim that, for every $\xi\in V$, the map $x\mapsto \rho(x)\xi$ is uniformly continuous with respect to the metric $d\big(\phi(\cdot),\phi(\cdot)\big)$. Since the linear span of $b[H]$ is dense in $V$, it suffices to prove this for $\xi\in b[H]$. So fix some $z\in H$ and note that, for $x,y\in H$, we have
\[\begin{split}
\Norm{\rho(x)b(z)-\rho(y)b(z)}_{\ku U,p}
&=\Norm{\big(b(xz)-b(x)\big)-\big(b(yz)-b(y)\big)}_{\ku U,p}\\
&\leqslant\Norm{b(xz)-b(yz)}_{\ku U,p}+\Norm{b(y)-b(x)}_{\ku U,p}\\
&\leqslant\theta\big[ d\big(\phi(xz), \phi(yz)\big)  \big]+\theta\big[ d\big(\phi(x),\phi(y)   \big)\big].
\end{split}\]
Now, suppose that $\eps>0$ is given. Then, by uniform continuity of $\sigma$, there is $\delta>0$ so that $\theta(\delta)<\frac \eps2$. Also, since multiplication by $\phi(z)$ on the right is left-uniformly continuous on $G$, there is an $\eta>0$ so that $d\big(\phi(x),\phi(y)   \big)<\eta$ implies that $d\big(\phi(xz),\phi(yz)   \big)<\delta$. It follows that, provided $d\big(\phi(x),\phi(y)   \big)<\min \{\eta,\delta\}$, we have $\Norm{\rho(x)b(z)-\rho(y)b(z)}_{\ku U,p}<\eps$, hence verifying uniform continuity.

To sum up, we have now an isometric linear representation $\rho\colon H\curvearrowright V$ with associated cocycle $b\colon H\til V$ so that, with respect to the metric $d\big(\phi(\cdot),\phi(\cdot)\big)$ on $H$, the mappings $b$ and $x\mapsto \rho(x)\xi$ are uniformly continuous for all $\xi\in V$.

Identifying $H$ with its image in $G$ via the continuous embedding $\phi$, it follows that there are unique continuous extensions of these mappings to all of $G$, which we continue denoting $b$ and $x\mapsto \rho(x)\xi$. Also, simple arguments using continuity and density show that, for all $x\in G$,  the map $\xi\mapsto \rho(x)\xi$ defines a linear isometry $\rho(x)$ of $V$  so that $\rho(xy)=\rho(x)\rho(y)$ and $b(xy)=\rho(x)b(y)+b(x)$. In other words, $\rho$ is a continuous isometric linear representation and $b\in Z^1(G,\rho)$. Moreover, 
$$
\tilde\kappa\big[d(x,y)\big]
\leqslant 
\Norm{b(x)-b(y)
}_{\ku U, p}\leqslant 
\theta\big[d(x,y)\big]
$$
for all $x,y\in G$.

To finish the proof, it now suffices to verify that $V$ is finitely representable in $L^p(E)$. To see this, note that, since each $L^p(F_n,E)$ is isometrically a subspace of $L^p(E)$, it follows from the properties of the ultraproduct that $\big(\prod_nL^p(F_n,E)\big)_\ku U$ is finitely representable in $L^p(E)$. Also, by construction, $X$ and a fortiori its subspace $V$ are isometrically embeddable into $\big(\prod_nL^p(F_n,E)\big)_\ku U$, which proves the theorem under the second assumption.

Consider now instead the case when $G$ is approximately compact. The proof is very similar to the second case, so we shall just indicate the changes needed. Thus, let $K_1\leqslant K_2\leqslant K_3\leqslant\ldots\leqslant G$ be a chain of compact subgroups with dense union in $G$. Instead of considering the sets $F_n$ with the Haar measure $|\cdot|$ from $H$, we now use the $K_n$ with their respective Haar measures and similarly $L^p(K_n,E)$ in place of $L^p(F_n,E)$. As before, the ultraproduct $\big(\prod_nL^p(K_n,E)\big)_\ku U$ is finitely representable in $L^p(E)$. 

Also, if $x\in \bigcup_nK_n$, then the right-regular representation $\rho(x)$ defines a linear isometry of all but finitely many $L^p(K_n,E)$, which means that $\rho(x)$ induces a linear isometry of the ultraproduct $\big(\prod_nL^p(K_n,E)\big)_\ku U$ simply by letting $\rho(x)\big(\xi_n\big)_\ku U=\big(\rho(x)\xi_n\big)_\ku U$.

Also, observe that, for $x,y\in G$, we have $\rho(x)\sigma,\rho(y)\sigma\in L^p(K_n,E)$ for all $n$ and by computations similar to those above, we find that
$$
\tilde\kappa\big[d\big(\phi(x),\phi(y)\big)\big]
\leqslant 
\NORM{
\big(\rho(x)\sigma-\rho(y)\sigma
\big)_\ku U
}_{\ku U}\leqslant 
\theta\big[d\big(\phi(x),\phi(y)\big)\big].
$$
We may therefore define  $b\colon G\til \big(\prod_nL^p(K_n,E)\big)_\ku U$ by letting 
$$
b(x)=\big(\rho(x)\sigma-\sigma\big)_\ku U
$$
and note that 
$$
\tilde\kappa\big[d\big(\phi(x),\phi(y)\big)\big]
\leqslant 
\Norm{
b(x)-b(y)
}_{\ku U}\leqslant 
\theta\big[d\big(\phi(x),\phi(y)\big)\big].
$$
One easily sees that, when restricted to $\bigcup_nK_n$, $b$ is a cocycle associated to the representation $\rho\colon \bigcup_nK_n\curvearrowright \big(\prod_nL^p(K_n,E)\big)_\ku U$.

Moreover, as before, we verify that, for every $\xi \in V=\ov{b[\bigcup_nK_n]}$, the mapping $x\in \bigcup_nK_n\mapsto \rho(x)\xi\in V$ is left-uniformly continuous, so the representation $\rho$ extends to an isometric linear representation of $G$ on $V$ with associated cocycle $b$. 
\end{proof}

Naor and Peres \cite{naor-peres} showed the above result for finitely generated discrete amenable groups using F\o lner sequences and asked in this connection whether one might bypass the F\o lner sequence of the proof and instead proceed directly from an invariant mean on the group. Of course, for a finitely generated group, this question is a bit vague since having an invariant mean or a F\o lner sequence are both equivalent to amenability. However, for general Polish groups this is not so and we may therefore provide a precise statement capturing the essence of their question.
\begin{prob}\label{naor}
Does Theorem \ref{pestov} hold for a general amenable Polish group $G$?
\end{prob}

\begin{rem} Subsequently to the appearence of the present paper in preprint, F. M. Schneider and A. Thom have developed a more general notion of F\o lner sets in topological groups \cite{thom} and were able to combine this with the mechanics of the above proof to solve Problem \ref{naor} in the affirmative. Moreover, they were also able  to construct an example of an amenable  Polish group, which is not F\o lner amenable, thereby showing that their result is indeed a proper strengthening of the above. 
\end{rem}

Our main concern here being the existence of coarsely proper affine isometric actions, let us first consider the application of Theorem \ref{pestov} to that problem.
\begin{cor}\label{fin repr}
Let $G$ be a F\o lner amenable Polish group admitting a uniformly continuous coarse embedding into a Banach space $E$. Then $G$ admits a coarsely proper continuous affine isometric action on a Banach space $V$ that is finitely representable in $L^2(E)$.
\end{cor}
Though we do not in general have tools permitting us to circumvent the assumption of uniform continuity as in Lemma \ref{randrianarivony}, for non-Archimedean groups we do. Indeed, assume $\sigma\colon G\til E$ is a coarse embedding of a non-Archimedean Polish group into a Banach  space $E$. Then $G$ is locally (OB) and thus has an open subgroup $V\leqslant G$ with property (OB) relative to $G$. Since $\sigma$ is bornologous,  there is a constant $K>0$ so that $\norm{\sigma(g)-\sigma(f)}\leqslant K$ whenever $f\inv g\in V$.
Letting $X\subseteq G$ denote a set of left-coset representatives for $V$, we define $\eta(g)=\sigma(h)$, where $h\in X$ is the coset representative of $gV$. Then $\norm{\eta(g)-\sigma(g)}\leqslant K$ for all $g\in G$, so $\eta\colon G\til \ku H$ is a coarse embedding and clearly constant on left-cosets of $V$, whence also uniformly continuous on $G$.
\begin{cor}\label{fin repr2}
Let $G$ be a F\o lner amenable non-Archimedean Polish group admitting a coarse embedding into a Banach space $E$. Then $G$ admits a coarsely proper continuous affine isometric action on a Banach space $V$ that is finitely representable in $L^2(E)$.
\end{cor}


Our second task is now to identify various properties of a Banach space $E$ that are inherited by any space $V$ finitely representable in $L^2(E)$ (or in other $L^p(E)$). Evidently, these must be {\em local properties} of Banach spaces, i.e., only dependent on the class of finite-dimensional subspaces of the space in question.

We recall that a Banach space $V$ is {\em super-reflexive} if every other space crudely finitely representable in $V$ is reflexive. Since evidently $V$ is finitely representable in itself, super-reflexive spaces are reflexive. Also, every space that is crudely finitely representable in a super-reflexive space must not only be reflexive but even super-reflexive. Moreover, super-reflexive spaces are exactly those all of whose ultrapowers are reflexive. By a result of P. Enflo \cite{enflo} (see also G. Pisier \cite{pisier} for an improved result or \cite{fabian} for a general exposition), the super-reflexive spaces can also be characterised as those admitting an equivalent uniformly convex renorming.

It follows essentially from the work of J. A. Clarkson \cite{clarkson} that, if $E$ is uniformly convex, then so is $L^p(E)$ for all $1<p<\infty$. In particular, if $E$ is super-reflexive, then so is every space finitely representable in $L^2(E)$.

For a second application, we shall note the preservation of Rademacher type and cotype in the above construction. For that we fix a {\em Rademacher sequence}, i.e., a sequence $(\eps_n)_{n=1}^\infty$ of mutually independent random variables $\eps_n\colon \Omega\til \{-1,1\}$, where $(\Omega, \mathbb P)$ is some probability space, so that $\mathbb P(\eps_n=-1)=\mathbb P(\eps_n=1)=\frac 12$. E.g., we could take $\Omega=\{-1,1\}^\N$ with the usual coin tossing measure and let $\eps_n(\omega)=\omega(n)$.
\begin{defi}
A Banach space $X$ is  said to have {\em type} $p$ for some $1\leqslant p\leqslant 2$ if there is a constant $C$ so that
$$
\Big(\mathbb E\NORM{\sum_{i=1}^n\eps_ix_i}^p\Big)^\frac 1p\leqslant C\cdot\Big(\sum_{i=1}^n\norm{x_i}^p\Big)^\frac 1p
$$
for every finite sequence $x_1,\ldots, x_n\in X$.

Similarly, $X$ has {\em cotype} $q$ for some $2\leqslant q<\infty$ if there is a constant $K$ so that 
$$
\Big(\sum_{i=1}^n\norm{x_i}^q\Big)^\frac 1q\leqslant K\cdot  \Big(\mathbb E\NORM{\sum_{i=1}^n\eps_ix_i}^q\Big)^\frac 1q
$$
for every finite sequence $x_1,\ldots, x_n\in X$.
\end{defi}
We note that, by the triangle inequality, every Banach space has type $1$. Similarly, by stipulation, every Banach space is said to have cotype $q=\infty$.

Whereas the $p$ in the formula $\Big(\sum_{i=1}^n\norm{x_i}^p\Big)^\frac 1p$ is essential, this is not so with the $p$ in $\Big(\mathbb E\NORM{\sum_{i=1}^n\eps_ix_i}^p\Big)^\frac 1p$. Indeed, the Kahane--Khintchine inequality (see \cite{albiac}) states that, for all $1<p<\infty$, there is a constant $C_p$ so that, for every Banach space $X$ and $x_1,\ldots, x_n\in X$, we have
$$
\mathbb E\NORM{\sum_{i=1}^n\eps_ix_i}
\leqslant 
\Big(\mathbb E\NORM{\sum_{i=1}^n\eps_ix_i}^p\Big)^\frac 1p
\leqslant 
C_p\cdot\mathbb E\NORM{\sum_{i=1}^n\eps_ix_i}.
$$
In particular, for any $p,q\in [1,\infty[$, the two expressions $\Big(\mathbb E\NORM{\sum_{i=1}^n\eps_ix_i}^p\Big)^\frac 1p$ and $\Big(\mathbb E\NORM{\sum_{i=1}^n\eps_ix_i}^q\Big)^\frac 1q$ differ at most by some fixed multiplicative constant independent of the space $X$ and the vectors $x_i\in X$. 

Clearly, if $E$ has type $p$ or cotype $q$, then so does every space crudely finitely representable in $X$.  Also, by results of W. Orlicz and G. Nordlander (see \cite{albiac}),  the space $L^p$ has type $p$ and cotype $2$, whenever $1\leqslant p\leqslant 2$, and type $2$ and cotype $p$, whenever $2\leqslant p<\infty$. Moreover, assuming again that $E$ has type $p$ or cotype $q$, then $L^2(E)$ has type $p$, respectively, cotype $q$.
We refer the reader to \cite{albiac} for more information on Rademacher type and cotype.

By the above discussion, Corollary \ref{fin repr} gives us the following.

\begin{cor}\label{type}
Let $G$ be a F\o lner amenable Polish group admitting a uniformly continuous coarse embedding into a Banach space $E$ that is either (a) super-reflexive, (b) has type $p$ or (c) cotype $q$. Then $G$ admits a coarsely proper continuous affine isometric action on a Banach space $V$ that is  super-reflexive, has type $p$, respectively, has cotype $q$.
\end{cor}
In particular, this applies when $G$ is a Banach space.  Also, in fact, any combination of properties (a), (b) and (c) verified by $E$ can be preserved by $V$.

Corollary \ref{type} applies, in particular, when $G$ admits a uniformly continuous coarse embedding into an $L^p$ space. In this connection, let us note that, by results of J. Bretagnolle, D. Dacunha-Castelle and J.-L. Krivine \cite{bretagnolle2, bretagnolle} and M. Mendel and Naor \cite{naor1}, for $1\leqslant q\leqslant p\leqslant \infty$, there is a map $\phi\colon L^q\til L^p$ which is simultaneously a  uniform and coarse embedding. Thus, if $G$ admits a uniformly continuous coarse embedding into $L^p$, $p\leqslant 2$, then it also admits such an embedding into $L^2$.
On the other hand, by results of Mendel and Naor \cite{naor2}, $L^q$ does not embed coarsely into $L^p$, whenever $\max \{2,p\}<q<\infty$.  

Using Theorems \ref{ell p sum} and \ref{pestov}, we conclude the following.
\begin{thm}\label{pestov banach}
Let  $\sigma\colon X\til E$ be an uncollapsed  uniformly continuous map between separable Banach spaces. Then, for every $1\leqslant p<\infty$, $X$ admits a coarsely proper continuous affine isometric action on a Banach space $V$  finitely representable in $L^p(E)$. 
\end{thm}

Also, similarly to the proof of Corollary \ref{balls hilbert}, we may conclude the following.
\begin{cor}\label{balls local}
Let $\sigma\colon X\til E$ be a uniformly continuous map from a Banach space $X$ into a super-reflexive or superstable space satisfying
$$
\inf_{\norm{x-y}=r}\norm{\sigma(x)-\sigma(y)}>0
$$
for just some $r>0$. Then $B_X$ is uniformly embeddable into a super-reflexive, respectively, superstable space. 
\end{cor}
Again, by Proposition 5.3 \cite{raynaud}, in the superstable case, we may further conclude that $X$ has finite cotype. In the super-reflexive case, we may combine this with a result of Kalton to conclude the following.
\begin{cor}\label{kalton-ros}
Let $\sigma\colon X\til E$ be a uniformly continuous map from a Banach space $X$ with non-trivial type into a super-reflexive space satisfying
$$
\inf_{\norm{x-y}=r}\norm{\sigma(x)-\sigma(y)}>0
$$
for just some $r>0$. Then $X$ is super-reflexive. 
\end{cor}
\begin{proof}
Indeed, by Theorem 5.1 \cite{kalton}, if $B_X$ is uniformly embeddable into a uniformly convex space or, equivalently, into a super-reflexive space, and $X$ has non-trivial type, then $X$ is super-reflexive. The result now follows by applying Corollary \ref{balls local}.
\end{proof}


\section{Stable metrics, wap functions  and reflexive spaces}
While it would be desirable to have a result along the lines of Theorems \ref{maurey} and \ref{pestov} for reflexive spaces, things are more complicated here and requires  auxiliary concepts, namely, stability and weakly almost periodic  functions. 

\begin{defi}
A  function $\Phi\colon X\times X\til \R$ on a set $X$ is {\em stable} provided that,  for all sequences $(x_n)$ and $(y_m)$ in $X$, we have
$$
\lim_{n\til \infty}\lim_{m\til \infty}\Phi(x_n,y_m)=\lim_{m\til \infty}\lim_{n\til \infty}\Phi(x_n,y_m),
$$
whenever both limits exist in $\R$.
\end{defi}
Observe that, if $\Phi$ is a metric or just an \'ecart on $X$, then, if either of the two limits exists, both $(x_n)$ and $(y_m)$ are bounded. Thus, for an \'ecart to be stable, it suffices to verify the criterion for bounded sequences $(x_n)$ and $(y_m)$. With this observation in mind, one can show that an \'ecart (respectively, a bounded function) $\Phi$ is stable provided that, for all bounded $(x_n)$, $(y_m)$ (respectively all $(x_n)$, $(y_m)$), we have 
$$
\lim_{n\til \ku U}\lim_{m\til \ku V}\Phi(x_n,y_m)=\lim_{m\til \ku V}\lim_{n\til \ku U}\Phi(x_n,y_m)
$$
whenever $\ku U, \ku V$ are non-principal ultrafilters on $\N$.

We remark that a simple, but tedious, inspection shows that, if $d$ is a stable \'ecart on $X$, then so is the uniformly equivalent bounded \'ecart $D(x,y)=\min\{d(x,y),1\}$. 

In keeping with the terminology above, a  Banach space $(X, \norm\cdot)$ is {\em stable} if the norm metric is stable. This class of spaces was initially studied by J.-L. Krivine and Maurey \cite{KM} in which they showed that every stable  Banach space contains a copy of $\ell^p$ for some $1\leqslant p<\infty$. As we shall see, for several purposes including the geometry of Banach spaces, the following more general class of groups plays a central role.
\begin{defi}
A topological group $G$ is {\em metrically stable} if $G$ admits a compatible left-invariant stable metric.
\end{defi}

Recall that a bounded  function $\phi\colon G\til \R$ on a group $G$ is said to be {\em weakly almost periodic (WAP)} provided that its orbit $\lambda(G)\phi=\{\phi(g\inv\,\cdot\, )\del g\in G\}$ under the left regular representation $\lambda\colon G\curvearrowright \ell^\infty(G)$ is a relatively weakly compact subset of $\ell^\infty(G)$.  
The connection between stability and weak almost periodicity is provided by the following well-known criterion of A. Grothendieck \cite{grothendieck}.
\begin{thm}[A. Grothendieck \cite{grothendieck}]
A  bounded  function $\phi\colon G\til \R$ on a group $G$ is weakly almost periodic if and only if the function $\Phi(x,y)=\phi(x\inv y)$ is stable.
\end{thm}
There is a tight connection between weakly almost periodic functions and representations on reflexive Banach space borne out by several classical results. Central among these is the fact that every continuous bounded weakly almost periodic function $\phi$ on a topological group $G$ is a {\em matrix coefficient} of a strongly continuous isometric linear representation $\pi\colon G\curvearrowright E$ on a reflexive Banach space $E$, that is, for some $\xi \in E$ and $\eta^*\in E^*$,
$$
\phi(x)=\langle \pi(x)\xi, \eta^*\rangle.
$$
Conversely, every such matrix coefficient is weakly almost periodic.

We shall need some more recent results clarifying reflexive representability of Polish groups. Here a topological group $G$ is said to {\em admit a topologically faithful isometric linear representation on a reflexive Banach space $E$} if $G$ is isomorphic to a subgroup of the linear isometry group ${\rm Isom}(E)$ equipped with the strong operator topology. The following theorem combines results due to A. Shtern \cite{shtern}, M. Megrelishvili \cite{megrelishvili2} and I. Ben Yaacov, A. Berenstein and S. Ferri \cite{ben yaacov}.

\begin{thm}\cite{shtern,megrelishvili2, ben yaacov}\label{smben}
For a Polish group $G$, the following are equivalent,
\begin{enumerate}
\item $G$ is  metrically stable,
\item $G$ has a topologically faithful isometric linear representation on a separable reflexive Banach space,
\item for every identity neighbourhood $V\subseteq G$, there is a continuous  weakly almost periodic function $\phi\colon G\til [0,1]$ so that $\phi(1_G)=1$ and ${\rm supp}(\phi)\subseteq V$.
\end{enumerate}
\end{thm}
In particular, the isometry group of a separable reflexive Banach space is metrically stable.
Moreover,  if $E$ is a stable Banach space, then ${\rm Isom}(E)$ is metrically stable. Indeed, by Lemma \ref{cocycle isometry group}, ${\rm Isom}(E)$ admits a  cocycle $b\colon {\rm Isom}(E)\til \ell^2(E)$ which is a uniform embedding. As also $\ell^2(E)$ is stable, it follows that $d(x,y)=\norm{b(x)-b(y)}_{\ell^2(E)}$ is a compatible left-invariant stable metric on ${\rm Isom}(E)$.

Conversely, there are examples, such as the group ${\rm Homeo}_+[0,1]$ of increasing homeomorphisms of the unit interval \cite{megrelishvili1}, that admit no non-trivial continuous isometric linear actions on a reflexive space.

We now aim at extending earlier results of Y. Raynaud \cite{raynaud} on the existence of $\ell^p$ subspaces of Banach spaces. Observe first that every left-invariant compatible metric on a Banach space is in fact bi-invariant and uniformly equivalent to the norm metric. Thus, a Banach space is metrically stable exactly when it admits an invariant stable metric uniformly equivalent to the norm. Formulated in our terminology, Th\'eor\`eme 4.1 \cite{raynaud} states the following.
\begin{thm}[Y. Raynaud]\label{raynaud}
Every metrically stable Banach space contains an isomorphic copy of $\ell^p$ for some $1\leqslant p<\infty$.
\end{thm}
 Also, by Theorem \ref{smben} and Proposition \ref{cocycle amplification banach}, we see that a metrically stable separable Banach space admits a coarsely proper continuous affine isometric action on a separable reflexive Banach space.
 
Furthermore, in Th\'eor\`eme 0.2 \cite{raynaud}, Raynaud showed that any Banach space uniformly embeddable in a {\em superstable} Banach space, i.e., a space all of whose ultrapowers are stable, is metrically stable (in fact, that it has a compatible invariant superstable metric).

Note finally that, since by  \cite{raynaud} the class of superstable Banach spaces is closed under passing from $E$ to $L^p(E)$, $1\leqslant p<\infty$, and under finite representability, Theorem \ref{pestov banach} implies that a separable Banach space $E$ uniformly embeddable into a superstable Banach space also admits a coarsely proper continuous cocycle on a superstable space.

\begin{thm}\label{embedding super-reflexive}
Let $X$ be a separable Banach space admitting an uncollapsed uniformly continuous map into the unit ball $B_E$ of a super-reflexive Banach space $E$. Then $X$ is metrically stable and contains an isomorphic copy of $\ell^p$ for some $1\leqslant p<\infty$. 
\end{thm}

\begin{proof}Let $\sigma\colon X\til B_E$ be the map in question.
By replacing $E$ by the closed linear span of the image of $Y$, we can suppose that $E$ is separable. Let $\kappa$ and $\theta$ be the compression and  expansion moduli of $\sigma$. As $\sigma$ is uncollapsed,  $\kappa(t)>0$ for some $t>0$ and, as $\sigma$ maps into $B_E$, we have $\theta(s)\leqslant 2$ for all $s$. By Theorem
\ref{pestov}, there is a Banach space $V$ finitely representable in $L^2(E)$ and a strongly continuous isometric linear representation $\pi\colon X\curvearrowright V$ with a continuous cocycle $b\colon X\til V$ satisfying 
$$
\kappa\big(\norm{x-y}\big)\leqslant \norm{b(x)-b(y)}_V\leqslant \theta\big(\norm{x-y}\big).
$$
Being finitely representable in the super-reflexive space $L^2(E)$, it follows that $V$ is super-reflexive itself. Also, $b$ is bounded and uncollapsed, whence, by Proposition \ref{affine banach}, $b$ is a uniform embedding of $X$ into $V$. 
Being a bounded cocycle in a super-reflexive space, we conclude by the Ryll-Nardzewski fixed point theorem that $b$ is a coboundary, i.e., that $b(x)=\xi-\pi(x)\xi$ for some $\xi\in V$ and all $x\in X$.  

Now, since $x\mapsto \xi-\pi(x)\xi$ is a uniform embedding, so is $x\mapsto \pi(x)\xi$, which shows that $\pi\colon X\til {\rm Isom}(V)$ is a topologically faithful isometric linear representation on a reflexive Banach space. By Theorem \ref{smben}, we conclude that $X$ is metrically stable. Finally, by Theorem \ref{raynaud}, $X$ contains an isomorphic copy of $\ell^p$ for some $1\leqslant p<\infty$. 
\end{proof}

A question first raised by Aharoni \cite{aharoni2} (see also the discussion in Chapter 8 \cite{lindenstrauss}) is to determine the class of Banach spaces $X$ uniformly embeddable into their unit ball $B_X$. It follows from  \cite{aharoni} that $c_0$ embeds uniformly into $B_{c_0}$ and, in fact, every separable Banach space $X$ containing $c_0$ embeds uniformly into its ball $B_X$.  Similarly, for $1\leqslant p\leqslant 2$, $L^p([0,1])$ is uniformly embeddable into the unit ball of $L^2([0,1])$ \cite{maurey}. Since also, for all $1\leqslant p<\infty$, the unit balls of $L^p([0,1])$ and $L^2([0,1])$ are uniformly homeomorphic by a result of E. Odell and Th. Schlumprecht \cite{distortion}, it follows that $L^p([0,1])$ in uniformly embeddable into the ball $B_{L^p([0,1])}$ for $1\leqslant p\leqslant 2$. This fails for $p>2$ by results of \cite{maurey}.

\begin{cor}
Let $E$ be a separable super-reflexive Banach space not containing $\ell^p$ for any $1\leqslant p<\infty$ or, more generally, which is not metrically stable. Then $E$ is not uniformly embeddable into $B_E$.
\end{cor}

For example, this applies to the $2$-convexification $T_2$ of the Tsirelson space and to V. Ferenczi's uniformly convex HI space \cite{ferenczi}.


One may  wonder whether Theorem \ref{embedding super-reflexive} has an analogue for uniform embeddings into super-reflexive spaces as opposed to their balls. I.e., if $X$ is an infinite-dimensional Banach space uniformly embeddable into a super-reflexive space, does it follow that $X$ contains an isomorphic copy of $\ell^1$ or an infinite-dimensional super-reflexive subspace? However, as shown by B. Braga (Corollary 4.15 \cite{braga3}), Tsirelson's space $T$ uniformly embeds into the super-reflexive space $(T_2)_{T_2}$ without, of course, containing $\ell^1$ or a super-reflexive subspace. The best one might hope for is thus some asymptotic regularity property of $X$.


We shall now consider the existence of coarsely proper continuous affine isometric actions on reflexive spaces by potentially non-amenable groups.
For locally compact second countable groups, such actions always exist, since N. Brown and E. Guentner \cite{BG} showed that every countable discrete group admits a proper affine isometric action on a reflexive Banach space and U. Haagerup and  A. Przybyszewska \cite{haagerup-affine} generalised this to locally compact second countable groups.
Also, Kalton \cite{kalton} showed that every stable metric space may be coarsely embedded into a reflexive Banach space, while, on the contrary, the Banach space $c_0$ does not admit a coarse embedding into a reflexive Banach space.

Our goal here is to provide an equivariant counter-part of Kalton's theorem.

\begin{thm}\label{stable metric refl}
Suppose a topological group $G$ carries a continuous left-invariant coarsely proper stable \'ecart. Then $G$ admits a coarsely proper continuous affine isometric action on a reflexive Banach space.
\end{thm}

Theorem \ref{stable metric refl} is a corollary of the following more detailed result. 

\begin{thm}\label{wap refl}
Suppose $d$ is a continuous left-invariant \'ecart on a topological group $G$ and assume  that, for all $\alpha>0$, there is a continuous weakly almost periodic function $\phi\in \ell^\infty(G)$  with $d$-bounded support so that $\phi\equiv 1$ on $D_\alpha=\{g\in G\del d(g,1)\leqslant \alpha\}$.

Then $G$ admits a continuous  isometric action $\pi\colon G\curvearrowright X$ on a reflexive Banach space and a continuous  cocycle $b\colon G\til X$ that is coarsely proper with respect to the \'ecart $d$.
\end{thm}

\begin{proof}Under the given assumptions, we claim that, for every integer $n\geqslant 1$,  there is a  continuous weakly almost periodic function $0\leqslant \phi_n\leqslant 1$ on $G$ so that
\begin{enumerate}
\item $\norm{\phi_n}_\infty=\phi(1_G)=1$,
\item$\norm{\phi_n-\lambda(g)\phi_n}_\infty\leqslant \frac 1{4^n} \text{ for all } g\in D_n$ and 
\item ${\rm supp}(\phi_n)$ is $d$-bounded. 
\end{enumerate}

To see this, we pick inductively sequences of  continuous weakly almost periodic functions $(\psi_i)_{i=1}^{4^n}$ and radii $(r_i)_{i=0}^{4^n}$ so that
\begin{enumerate}
\item[(i)] $0=r_0<2n<r_1<r_1+2n<r_2<r_2+2n<r_3<\ldots<r_{4^n}$,
\item[(ii)] $0\leqslant \psi_i\leqslant 1$,
\item[(iii)] ${\psi}_{i}\equiv 1$ on $D_{r_{i-1}+n}$,
\item[(iv)] ${\rm supp}(\psi_i)\subseteq D_{r_i}$.
\end{enumerate}
Note first that, by the choice of $r_i$, the sequence
$$
D_{r_0+n}\setminus D_{r_0},\; D_{r_1}\setminus D_{r_0+n}, \;D_{r_1+n}\setminus D_{r_1},\; D_{r_2}\setminus D_{r_1+n}, \; \ldots\;, D_{r_{4^n}}\setminus D_{r_{4^n-1}+n}, \; G\setminus D_{r_{4^n}}
$$
partitions $G$. Also, for all $1\leqslant i\leqslant 4^n$, 
$$
\psi_1\equiv \ldots\equiv  \psi_i\equiv 0, \text{ while } \psi_{i+1}\equiv \ldots\equiv \psi_{4^n}\equiv 1\text{ on }D_{r_i+n}\setminus D_{r_i}
$$
and 
$$
\psi_1\equiv \ldots\equiv \psi_{i-1}\equiv 0, \text{ while } \psi_{i+1}\equiv \ldots\equiv  \psi_{4^n}\equiv 1\text{ on }D_{r_i}\setminus D_{r_{i-1}+n}.
$$
Setting  $\phi_n=\frac 1{4^n}\sum_{i=1}^{4^n}\psi_i$, we note that, for  all $1\leqslant i\leqslant 4^n$, 
$$
\phi_n\equiv \frac {4^n-i}{4^n} \quad \text{ on }D_{r_i+n}\setminus D_{r_i}
$$
and
$$
 \frac {4^n-i}{4^n}\leqslant \phi_n\leqslant  \frac {4^n-i+1}{4^n} \quad \text{ on }D_{r_{i}}\setminus D_{r_{i-1}+n}.
$$

Now, if $g\in D_n$ and $f\in G$, then $|d(g\inv f,1)-d(f,1)|=|d(f,g)-d(f,1)|\leqslant d(g,1)\leqslant n$. So, if $f$ belongs to some term in the above partition, then $g\inv f$ either belongs to the immediately preceding, the same or the immediately following term of the partition. By the above estimates on $\phi_n$, it follows that $|\phi_n(f)-\phi_n(g\inv f)|\leqslant \frac1{4^n}$.
In other words, for $g\in D_n$, we have
\[\begin{split}
\norm{\phi_n-\lambda(g)\phi_n}_\infty= \sup_{f\in G}|\phi_n(f)-\phi_n(g\inv f)|\leqslant \frac 1{4^n},
\end{split}\]
which verifies condition (2). Conditions (1) and (3) easily follow from the construction.

Consider now a specific $\phi_n$ as above and define 
$$
W_n=\ov{\rm conv}\big(\lambda(G)\phi_n\cup -\lambda(G)\phi_n\big)\subseteq B_{\ell^\infty(G)}
$$
and, for every $k\geqslant 1$, 
$$
U_{n,k}=2^kW_n+2^{-k}B_{\ell^\infty(G)}.
$$
Let $\norm{\cdot}_{n,k}$ denote the gauge on ${\ell^\infty}(G)$ defined by $U_{n,k}$, i.e., 
$$
\norm{\psi}_{n,k}=\inf(\alpha>0\del \psi\in \alpha\cdot U_{n,k}).
$$

If $g\in D_n$, then $\norm{\phi_n-\lambda(g)\phi_n}_\infty\leqslant \frac 1{4^n}$ and so, for $k\leqslant n$, 
$$
\phi_n-\lambda(g)\phi_n\in \frac 1{2^n}\cdot 2^{-k}B_{\ell^\infty(G)}\subseteq \frac 1{2^n}\cdot U_{n,k}.
$$
In particular, 
\begin{equation}\label{a}
\norm{\phi_n-\lambda(g)\phi_n}_{n,k}\leqslant \frac 1{2^n},\;\;\text{ for all } k\leqslant n \text{ and }g\in D_n.
\end{equation}

On the other hand,  for all $g\in G$ and $k$, we have $\phi_n-\lambda(g)\phi_n\in 2W_n\subseteq \frac1{2^{k-1}}U_{n,k}$. Therefore,
\begin{equation}\label{b}
\norm{\phi_n-\lambda(g)\phi_n}_{n,k}\leqslant \frac 1{2^{k-1} },\;\;\text{ for all } k\text{ and }g.
\end{equation}

Finally, since $U_{n,1}\subseteq 2B_{\ell^\infty(G)}+\frac 12B_{\ell^\infty(G)}=\frac 52B_{\ell^\infty(G)}$, we have $\norm\cdot_\infty\leqslant \frac 52\norm\cdot_{n,1}$. So,  if $g\notin ({\rm supp}\;\phi_n)\inv$, then 
\begin{equation}\label{c}
\norm{\phi_n-\lambda(g)\phi_n}_{n,1}\geqslant \frac 25,\;\;\text{ for all } g\notin ({\rm supp}\;\phi_n)\inv.
\end{equation}

It follows from (\ref{a}) and (\ref{b}) that, for $g\in D_n$, we have 
\[\begin{split}
\sum_k\norm{\phi_n-\lambda(g)\phi_n}^2_{n,k}
&\leqslant \underbrace{\Big(\frac 1{2^n}\Big)^2+\ldots+\Big(\frac 1{2^n}\Big)^2}_{n\text{ times}}
+ \Big(\frac 1{2^{(n+1)-1}}\Big)^2+ \Big(\frac 1{2^{(n+2)-1}}\Big)^2+\ldots\\
&\leqslant\frac 1{2^{n-2}},
\end{split}\]
while using (\ref{c}) we have, for $g\notin ({\rm supp}\;\phi_n)\inv$, 
$$
\sum_k\norm{\phi_n-\lambda(g)\phi_n}^2_{n,k}\geqslant \frac4{25}.
$$

Define $\triple{\cdot}_n$ on $\ell^\infty(G)$ by $\triple{\psi}_n=\big(\sum_k\norm{\psi}_{n,k}^2\big)^{\frac 12}$ and set
$$
X_n=\{\psi\in \ov{\rm span}(\lambda(G)\phi_n)\subseteq \ell^\infty(G)\del \triple{\psi}_n<\infty\}\subseteq \ell^\infty(G).
$$
By the main result of  W. J. Davis, T. Figiel, W. B. Johnson and A. Pe\l czy\'nski \cite{dfjp}, the interpolation space $(X_n,\triple{\cdot}_n)$ is a  reflexive Banach space.  Moreover, since $W_n$ and $U_{n,k}$ are $\lambda(G)$-invariant subsets of $\ell^\infty(G)$, one sees that $\norm{\cdot}_{n,k}$ and $\triple{\cdot}_n$ are $\lambda(G)$-invariant and hence we have an isometric linear representation $\lambda\colon G\curvearrowright (X_n, \triple{\cdot}_n)$.

Note that, since $\phi_n\in W_n$, we have $\phi_n\in X_n$ and can therefore define a cocycle $b_n\colon G\til X_n$ associated to $\lambda$ by $b_n(g)=\phi_n-\lambda(g)\phi_n$. 
By the estimates above, we have
$$
\triple{b_n(g)}_n=\triple{\phi_n-\lambda(g)\phi_n}_n\leqslant\big( \frac 1{{\sqrt 2}}\big)^{n-2}
$$ 
for $g\in D_n$, while
$$
\triple{b_n(g)}_n=\triple{\phi_n-\lambda(g)\phi_n}_n\geqslant\frac2{5}
$$
for $g\notin ({\rm supp}\;\phi_n)\inv$.

Let now $Y=\big(\bigoplus_n(X_n,\triple{\cdot}_n)\big)_{\ell^2}$ denote the  $\ell^2$-sum of the spaces $(X_n,\triple{\cdot}_n)$. Let also $\pi\colon G\curvearrowright Y$ be the diagonal action and $b=\bigoplus b_n$ the corresponding cocycle. To see that $b$ is well-defined, note that, for $g\in D_n$, we have
\[\begin{split}
\norm{b(g)}_Y
&=\Big(\sum_{m=1}^\infty\triple{b_m(g)}_m^2\Big)^\frac 12\\
&=\Big(\text{finite}+\sum_{m=n}^\infty\triple{b_m(g)}_m^2\Big)^\frac 12\\
&\leqslant\Big(\text{finite}+\sum_{m=n}^\infty\frac 1{2^{m-2}}\Big)^\frac 12\\
&<\infty,
\end{split}\]
so $b(g)\in Y$.

Remark that, whenever $g\notin ({\rm supp}\;\phi_n)\inv$, we have
$$
\norm{b(g)}_Y\geqslant \Big(  \underbrace{\big(\frac 25\big)^2+\ldots+\big(\frac 25\big)^2}_{n\text{ times}}  \Big)^\frac 12\geqslant \frac {\sqrt n}3.
$$  
As $({\rm supp}\;\phi_n)\inv$ is $d$-bounded, this 
shows that the cocycle $b\colon G\til Y$ is coarsely proper with respect to $d$. We leave the verification that the action is continuous to the reader.
\end{proof}

Let us now see how to deduce Theorem \ref{stable metric refl} from Theorem \ref{wap refl}. So fix a continuous left-invariant coarsely proper stable \'ecart $d$ on $G$ with corresponding balls $D_\alpha$. Then, for every $\alpha>0$, we can define a continuous bounded weakly almost periodic function $\phi_\alpha\colon G\til \R$ by
$$
\phi_\alpha(g)=2-\max\Big\{1, \min\big\{\frac{d(g,1)}\alpha, 2\big\}\Big\}.
$$
We note that $\phi_\alpha$ has $d$-bounded support, while $\phi_\alpha\equiv 1$ on $D_\alpha$, thus verifying the conditions of Theorem \ref{wap refl}.


\section{A fixed point property for affine isometric actions}

\subsection{Kalton's theorem and solvent maps}\label{section solvent}
The result of this final section has a different flavor from the preceding ones, in that our focus will be on the interplay between coarse geometry and harmonic analytic properties of groups as related to fixed points of affine isometric actions.  As a first step we must consider a weakening of the concept of coarse embeddings and show how it relates to work of Kalton.

If $M\subseteq \N$ is an infinite set and $r\geqslant 1$, let $P_r(M)$ be the set of all $r$-tuples $n_1<\ldots<n_r$ with $n_i\in M$. We define a graph structure on $P_r(M)$ by letting two $r$-tuples $n_1<\ldots< n_r$ and $m_1<\ldots<m_r$ be related by an edge if
$$
n_1\leqslant m_1\leqslant n_2\leqslant m_2\leqslant \ldots\leqslant n_r\leqslant m_r
$$
or vice versa. When equipped with the induced path metric, one sees that $P_r(M)$ is a finite diameter metric space with the exact diameter depending on $r$.

\begin{thm}[N. J. Kalton \cite{kalton}]\label{kalton1}
Suppose $r\in \N$ and let $E$ be a Banach space such that $E^{(2r)}$ is separable. Then given any uncountable family $\{f_i\}_{i\in I}$ of bounded maps $f_i\colon P_r(\N)\til E$ and any $\epsilon>0$, there exist $i\neq j$ and an infinite subset $M\subseteq \N$ so that
$$
\norm{f_i(\sigma)-f_j(\sigma)}<\theta_{f_i}(1)+\theta_{f_j}(1)+\eps
$$
for all $\sigma\in P_r(M)$.
\end{thm}

Kalton uses this result to show that $c_0$ neither embeds uniformly nor coarsely into a reflexive Banach space. However a close inspection of the proof reveals a stronger result in the coarse setting.

\begin{defi}
A map $\phi\colon (X,d)\til (M,\ku E)$ from a metric space $X$ to a coarse space $M$ is said to be {\em solvent} if, for every coarse entourage $E\in \ku E$ and $n\geqslant 1$, there is a constant $R$ so that
$$
R\leqslant d(x,y)\leqslant R+n\;\;\saa\;\; \big(\phi(x),\phi(y)\big)\notin E.
$$
\end{defi}
In the case $(M,\partial)$ is a metric space, our definition becomes more transparent. Indeed $\phi$ is solvent if there are constants $R_n$ for all $n\geqslant 1$ so that 
$$
R_n\leqslant d(x,y)\leqslant R_n+n\;\;\saa\;\; \partial\big(\phi(x),\phi(y)\big)\geqslant n.
$$
Also, in case $X$ is actually a geodesic metric space, we have the following easy reformulation.
\begin{lemme}\label{solvent reform}
Let $(X,d)$ be a geodesic metric space of infinite diameter and suppose that $\phi\colon (X,d)\til (M,\partial)$ is a bornologous map into a metric space $(M,\partial)$. Then $\phi$ is solvent if and only if
$$
\sup_{t}\tilde\kappa_\phi(t)=\sup_{t}\inf_{d(x,y)=t}\partial\big(\phi(x),\phi(y)\big)=\infty
$$
\end{lemme}

\begin{proof}
Suppose that the second condition holds and find constants $t_n$ with 
$$
\inf_{d(x,y)=t_n}\partial\big(\phi(x),\phi(y)\big)\geqslant n.
$$ 
Since $\phi$ is bornologous and $X$ geodesic, $\phi$ is Lipschitz for large distances and hence
$$
\partial\big(\phi(x),\phi(y)\big)\leqslant K\cdot d(x,y)+K
$$
for some constant $K$ and all $x,y\in X$.

Now, assume $t_{n^2}\leqslant d(x,y)\leqslant t_{n^2}+n$. Then, as $X$ is geodesic, there is some $z\in X$ with $d(x,z)=t_{n^2}$, while $d(z,y)\leqslant n$. It follows that $\partial\big(\phi(x),\phi(z)\big)\geqslant n^2$, while $\partial\big(\phi(z),\phi(y)\big)\leqslant Kn+K$, i.e., 
$$
\partial\big(\phi(x),\phi(y)\big)\geqslant n^2-Kn-K\geqslant n
$$
provided $n$ is sufficiently large. Setting $R_n=t_{n^2}$, we see that $\phi$ is solvent.

The converse is proved by noting that every distance is attained in $X$.
\end{proof}

While a solvent map $\phi\colon X\til M$ remains solvent when post-composed with a coarse embedding of $M$ into another coarse space, then dependence on $X$ is somewhat more delicate. For this and ulterior purposes, we need the notion of near isometries.
\begin{defi} 
A map $\phi\colon X\til Y$ between metric spaces $X$ and $Y$ is said to be a {\em near isometry} if there is a constant $K$ so that, for all $x,x'\in X$, 
$$
d(x,x')-K\leqslant d(\phi(x),\phi(x'))\leqslant d(x,x')+K.
$$
\end{defi}
It is fairly easy to find a near isometry $\phi\colon X\til X$ of a metric space $X$ that is not close to any isometry of $X$, i.e., so that $\sup_xd(\phi(x),\psi(x))=\infty$ for any isometry $\psi$ of $X$. However, in the case of Banach spaces, positive results do exist.  Indeed, by the work of several authors, in particular, P. M. Gruber \cite{gruber} and J. Gervirtz \cite{gervirtz}, if $\phi\colon X\til Y$ is a surjective near isometry between Banach spaces $X$ and $Y$ with defect $K$ as above and $\phi(0)=0$, then there is surjective linear isometry $U\colon X\til Y$ with $\sup_{x\in X}\norm{\phi(x)-U(x)}\leqslant 4K$ (see Theorem 15.2 \cite{lindenstrauss}). 

Now, suppose instead that $\phi$ neither surjective nor is $\phi(0)=0$, but only that $\phi$ $K$-cobounded, i.e.,  $\sup_{y\in Y}\inf_{x\in X}\norm{y-\phi(x)}\leqslant K$. Then $Y$ has density character at most that of $X$ and hence cardinality at most that of $X$, whence a short argument shows that there is a surjective map $\psi\colon X\til Y$ with $\psi(0)=0$ so that $\sup_{x\in X}\norm{\psi(x)+\phi(0)-\phi(x)}\leqslant 42K$. In particular, $\psi$ is a near isometry with defect $85K$, whence there is a surjective linear isometry $U\colon X\til Y$ with $\sup_{x\in X}\norm{\psi(x)-U(x)}\leqslant 340K$. All in all, we find that
$$
\sup_{x\in X}\norm{A(x)-\phi(x)}\leqslant 382K,
$$
where $A\colon X\til Y$ is the surjective affine isometry $A=U+\phi(0)$. In other words, every cobounded near isometry $\phi\colon X\til Y$ is close to a surjective affine isometry $A\colon X\til Y$.

Let us also note the following  straightforward fact.
\begin{lemme}\label{coarse dependence}
Let 
$$
X\overset{\sigma}\longrightarrow Y\overset{\phi}\longrightarrow Z \overset{\psi}\longrightarrow W
$$
be maps between metric spaces $X$, $Y$ and coarse spaces $Z$, $W$ so that $\sigma$ is a near isometry, $\phi$ is solvent and $\psi$ is a coarse embedding. Then the composition $\psi\phi\sigma\colon X\til W$ is also solvent.
\end{lemme}

To gauge of how weak the existence of solvable maps is, we may reutilse the ideas of Section \ref{uniform vs coarse}.
\begin{prop}\label{triviality}
Supppose $X$ and $E$ are Banach spaces and that there is no uniformly continuous solvent map $\psi\colon X\til \ell^2(E)$. Then, for every uniformly continuous map $\phi\colon X\til E$, we have 
$$
\sup_r\inf_{\norm{x-y}=r}\Norm{\phi(x)-\phi(y)}=0.
$$
\end{prop}

\begin{proof}
Assume for a contradiction that $\phi\colon X\til E$ is a uniformly continuous map and that $r>0$ satisfies $\inf_{\norm{x-y}=r}\Norm{\phi(x)-\phi(y)}=\delta>0$. Without loss of generality, we may assume that $\phi(0)=0$.

Now fix $n\geqslant 1$ and choose $\eps_n>0$ small enough that $\theta_\phi(\eps_n)\leqslant \frac \delta{n2^n}$, where $\theta_\phi$ is the expansion modulus of $\phi$. Set also $\psi_n(x)=\frac n\delta\phi\big(\frac {\eps_n} n x\big)$. Then
\[\begin{split}
\norm{x-y}\leqslant n 
&\;\saa\; \NORM{\phi\Big(\frac {\eps_n} nx\Big)-\phi\Big(\frac {\eps_n} ny\Big)}\leqslant \frac \delta{n2^n}\\
&\;\saa\; \norm{\psi_n(x)-\psi_n(y)}\leqslant \frac 1{2^n},
\end{split}\]
while
\[\begin{split}
\norm{x-y}=\frac {rn}{\eps_n} &\;\saa\; \norm{\psi_n(x)-\psi_n(y)}\geqslant n.
\end{split}\]

Finally, let $\psi\colon X\til \ell^2(E)$ be defined by $\psi(x)=(\psi_1(x), \psi_2(x),\ldots)$. Then the above inequalities show that $\psi$ is well-defined and uniformly continuous and also that, for every $n\geqslant 1$, there is some $R$ with
$$
\norm{x-y}=R \;\saa\; \norm{\psi_n(x)-\psi_n(y)}_E\geqslant n\;\saa\; \norm{\psi(x)-\psi(y)}_{\ell^2(E)}\geqslant n.
$$
Applying Lemma \ref{solvent reform}, we conclude that $\psi$ is solvent, contradicting our assumptions.
\end{proof}

We now come to the improved statement of Kalton's theorem.
\begin{thm}\label{kalton2}
Every bornologous map $\phi\colon c_0\til E$  into a reflexive Banach space $E$ is insolvent. 
\end{thm}

\begin{proof}
Let $(e_k)_{k=1}^\infty$ denote the canonical unit vector basis for $c_0$. If $a\subseteq \N$ is a finite subset, we let $\chi_a=\sum_{k\in a}e_k$.  Now, for $r\geqslant 1$ and $A\subseteq \N$ infinite, define
$$
f_{r,A}(n_1,\ldots, n_r)=\sum_{i=1}^r\chi_{A\cap[1,n_i]}.
$$
Note that, if $n_1\leqslant m_1\leqslant n_2\leqslant m_2\leqslant \ldots\leqslant n_r\leqslant m_r$ with $n_1<n_2\ldots< n_r$ and $m_1<m_2<\ldots<m_r$, then $\norm{f_{r,A}(n_1,\ldots, n_r)-f_{r,A}(m_1,\ldots, m_r)}\leqslant 1$, that is, $\theta_{f_{r,A}}(1)\leqslant 1$ and thus also $\theta_{\phi f_{r,A}}(1)\leqslant \theta_\phi(1)$.

Now fix $r$. Then by Theorem \ref{kalton1} there are  infinite subsets $A,B, M\subseteq \N$ so that
$$
\norm{\phi f_{r,A}(\sigma)-\phi f_{r,B}(\sigma)}<\theta_{\phi f_{r,A}}(1)+\theta_{\phi f_{r,B}}(1)+1\leqslant 2\,\theta_\phi(1)+1.
$$ 
for all $\sigma\in P_r(M)$. On the other hand, as $A\neq B$, there is some $\sigma\in P_r(M)$ so that $\norm{f_{r,A}(\sigma)-f_{r,B}(\sigma)}=r$, whereby
$$
\inf_{\norm{x-y}=r}\norm{\phi(x)-\phi(y)} \leqslant\norm{\phi f_{r,A}(\sigma)-\phi f_{r,B}(\sigma)}\leqslant 2\,\theta_\phi(1)+1.
$$
As any interval of the form $[R,R+n]$, for $n\geqslant 1$, contains an integral point $r$, we see that $\phi$ is insolvent.
\end{proof}

\begin{cor}\label{kalton3}
Every bornologous map $\phi\colon c_0\til L^1([0,1])$ is insolvent. 
\end{cor}

\begin{proof}
By Lemma \ref{coarse dependence}, it suffices to observe that $L^1([0,1])$ coarsely embeds into the reflexive space $L^2([0,1])$, which follows from \cite{bretagnolle2}.
\end{proof}

Applying Proposition \ref{triviality} and the fact that $E\mapsto \ell^2(E)$ preserves reflexivity, we also  obtain the following corollary.
\begin{cor}
Every uniformly continuous map $\phi\colon c_0\til E$  into a reflexive Banach space  satisfies
$$
\sup_r\inf_{\norm{x-y}=r}\Norm{\phi(x)-\phi(y)}=0.
$$
\end{cor}

For a thorough study of regularity properties of Banach spaces preserved under solvent maps, the reader may consult \cite{braga2}.


\subsection{Geometric Gelfand pairs}
Our next step is to identify a class of topological groups whose large scale geometry is partially preserved through images by continuous homomorphisms. This is closely related to the classical notion of Gelfand pairs.

\begin{defi}\label{geom gelfand}
A {\em geometric Gelfand pair} consists of a coarsely proper continuous isometric action $G\curvearrowright X$ of a topological group $G$ on a metric space $X$ so that, for some constant $K$ and all $x,y,z,u\in X$ with $d(x,y)\leqslant d(z,u)$, there is $g\in G$ with $d\big(g(x),z\big)<K$ and $d\big(z,g(y)\big)+d\big(g(y),u\big)<d(z,u)+K$.
\end{defi}

An immediate observation is that, if $G\curvearrowright X$ is a geometric Gelfand pair and $x,z\in X$, we may set
$x=y$ and $z=u$ and thus find some $g\in G$ so that $d\big(g(x),z)<K$. In other words, the action is also cobounded, whereby the orbit map $g\in G\mapsto g(x)\in X$ is a coarse equivalence between $G$ and $X$ for any choice of $x\in X$. 

To better understand the notion of geometric Gelfand pairs, suppose that $G\curvearrowright X$ is a coarsely proper continuous and {\em doubly transitive} isometric action on a geodesic metric space, i.e., so that, for all $x,y,z,v\in X$ with $d(x,y)=d(z,v)$, we have $g(x)=z$ and $g(y)=v$ for some $g\in G$. Then $G\curvearrowright X$ is a geometric Gelfand pair. Indeed, if $d(x,y)\leqslant d(z,u)$, pick by geodecity some $v\in X$ so that $d(z,v)+d(v,u)=d(z,u)$ and $d(z,v)=d(x,y)$. Then, by double transitivity, there is $g\in G$ with $g(x)=z$ and $g(y)=v$, verifying that this is a geometric Gelfand pair.

Observe here that it suffices that $X$ is geodesic with respect to its distance set $S=d(X\times X)$, i.e., that, for all $x,y\in X$ and $s\in S$ with $s\leqslant d(x,y)$, there is $z\in X$ with $d(x,z)+d(z,y)=d(x,y)$ and $d(x,z)=s$.

If $K$ is a compact subgroup of a locally compact group $G$, then $(G,K)$ is said to be a {\em Gelfand pair} if the convolution algebra of compactly supported $K$-bi-invariant continuous functions on $G$ is commutative (see Section 3.3 \cite{bekka}  or Section 24.8 \cite{simonnet}). A basic result due to I. M.  Gelfand (Proposition 24.8.1 \cite{simonnet}) states that, if $G$ admits a involutory automorphism $\alpha$ so that $g\inv \in K\alpha(g)K$ for all $g\in G$, then $(G,K)$ is a Gelfand pair. 

In the case above of a  coarsely proper continuous and doubly transitive isometric action $G\curvearrowright X$ on a geodesic metric space, let $K=\{g\in G\del g(x)=x\}$, which is a closed subgroup of $G$. Then, if $g,f\in G$ and $f\in KgK$, there is some $h\in K$ so that $f(x)=hg(x)$ and so 
$$
d(f(x),x)=d(hg(x),x)=d(g(x), h\inv(x))=d(g(x),x).
$$
Conversely, if $g,f\in G$ and $d(f(x),x)=d(g(x),x)$, then, by double transitivity, there is some $h\in G$ with $hf(x)=g(x)$ and $h(x)=x$, i.e., $g\inv hf\in K$ and $h\in K$ whereby $f\in KgK$. In other words, 
$$
f\in KgK\;\;\equi\;\; d(f(x), x)=d(g(x),x).
$$
As $d(g\inv(x), x)=d(x,g(x))$, this shows that $g\inv\in KgK$ for all $g\in G$. So $(G,K)$ fulfils the condition of Gelfand's result, except for the fact that $G$ and $K$  may not be locally compact, respectively, compact. This motivates the terminology of Definition \ref{geom gelfand}.

\begin{exa}
Probably the simplest non-trivial example of a geometric Gelfand pair is that of the canonical action $D_\infty \curvearrowright \Z$ of the infinite dihedral group $D_\infty$ on $\Z$ with the euclidean metric. This action is doubly transitive and proper. As moreover $\Z$ is geodesic with respect to its distance set $\N$, we see that  $D_\infty \curvearrowright \Z$ is a geometric Gelfand pair.
\end{exa}
\begin{exa}
Similarly, if ${\rm Aff}(\R^n)=O(n)\ltimes \R^n$ denotes the group of (necessarily affine) isometries of the $n$-dimensional euclidean space, then ${\rm Aff}(\R^n)\curvearrowright\R^n$ is a geometric Gelfand pair.
\end{exa}

Another class of examples are constructed as above from Banach spaces. For this, let $X$ be a separable Banach space and ${\rm Isom}(X)$ the group of linear isometries of $X$ equipped with the strong operator topology. As every isometry of $X$ is affine, the group of all isometries of $X$ decomposes as the semi-direct product ${\rm Aff}(X)={\rm Isom}(X)\ltimes (X,+)$.  Also, $(X,\norm\cdot)$ is said to be {\em almost transitive}  if the action ${\rm Isom}(X)\curvearrowright X$ induces a dense orbit on the unit sphere $S_X$ of $X$. 
\begin{prop}
Let $X$ be a separable Banach space. Then ${\rm Aff}(X)\curvearrowright X$ is a geometric Gelfand pair if and only if ${\rm Isom}(X)$ has property (OB) and $(X,\norm\cdot)$ is almost transitive.
\end{prop}
\begin{proof}
First, as shown in \cite{rosendal-coarse}, the tautological action ${\rm Aff}(X)\curvearrowright X$ is coarsely proper if and only if ${\rm Isom}(X)$ has property (OB). So assume that ${\rm Isom}(X)$ has property (OB) and consider the  transitivity condition.

We observe that, if  $(X,\norm\cdot)$ is almost transitive, then ${\rm Aff}(X)\curvearrowright X$ is almost doubly transitive in the sense that, for all $x,y,z,u\in X$ with $\norm{x-y}=\norm{z-u}$ and $\eps>0$ there is $g\in {\rm Aff}(X)$ so that $g(x)=z$ and $\norm{g(y)-u}<\eps$. In this case, ${\rm Aff}(X)\curvearrowright X$ is a geometric Gelfand pair. Conversely, suppose ${\rm Aff}(X)\curvearrowright X$ is a geometric Gelfand pair with a defect $K$. Then, for $y,z\in S_X$ and $\eps>0$, there is some $g\in {\rm Aff}(X)$ so that $\norm{g(0)-0}<K$ and $\norm{g(\frac{2K}{\eps}y)-\frac {2K}{\eps}z}<K$. Letting $f\in {\rm Isom}(X)$ be the linear isometry $f(x)=g(x)-g(0)$, we find that $\norm{f(\frac{2K}{\eps}y)-\frac {2K}{\eps}z}<2K$, i.e., $\norm{f(y)-z}<\eps$, showing that $(X,\norm\cdot)$ is almost transitive. 
\end{proof}
Examples of almost transitive Banach spaces include $L^p([0,1])$ for all $1\leqslant p<\infty$ \cite{rolewicz} and  the Gurarii space $\mathbb G$ \cite{lusky}. Moreover, as all of these spaces are separably categorical in the sense of continuous model theory, it follows from Theorem 5.2 \cite{OB} that their linear isometry group has property (OB). We therefore conclude that ${\rm Aff}(L^p([0,1]))\curvearrowright L^p([0,1])$ with $1\leqslant p<\infty$ and ${\rm Aff}(\mathbb G)\curvearrowright\mathbb G$ are geometric Gelfand pairs.

\begin{exa}
Suppose $\Gamma$ is a countable {\em metrically homogeneous} connected graph, that is, so that any isometry $f\colon A\til B$ with respect to the path metric on $\Gamma$ between two finite subsets extends to a full automorphism of $\Gamma$. Then the action ${\rm Aut}(\Gamma)\curvearrowright \Gamma$ is coarsely proper \cite{rosendal-coarse} and metrically doubly transitive. Since also the path metric is $\Z$-geodesic, it follows that ${\rm Aut}(\Gamma)\curvearrowright \Gamma$ is a geometric Gelfand pair. 

Examples of metrically homogeneous countable graphs include the $n$-regular trees $T_n$ for all $1\leqslant n\leqslant \aleph_0$ and the integral Urysohn metric space $\Z\U$.
\end{exa}

This latter may be described as the Fra\"iss\'e limit of all finite $\Z$-metric spaces (i.e., with integral distances) and plays the role of a universal object in the category of $\Z$-metric spaces.  Alternatively, $\Z\U$ is the unique universal countable $\Z$-metric space so that any isometry between two finite subspaces extends to a full isometry of $\Z\U$.  The rational Urysohn metric space $\Q\U$ is described similarly with $\Q$ in place of $\Z$.

Let us also recall that, since  $\Q\U$ is countable,   
${\rm Isom}(\Q\U)$ is a Polish group when equipped with the {\em permutation group topology}, which is that obtained by declaring pointwise stabilisers to be open.

\begin{exa}
As shown in \cite{autom}, the tautological action ${\rm Isom}(\Q\U)\curvearrowright \Q\U$ is coarsely proper. Since it is also doubly transitive and $\Q\U$ is $\Q$-geodesic, it follows that ${\rm Isom}(\Q\U)\curvearrowright \Q\U$ is a geometric Gelfand pair.
\end{exa}


\subsection{Nearly isometric actions and quasi-cocycles}
Having identified the class of geometric Gelfand pairs, we will now show a strong geometric rigidity property for their actions. 

\begin{defi}
Let $G$ be a group and $M$ a metric space. A {\em nearly isometric action} of $G$ on $M$ is a map $\alpha\colon G\times M\til M$ so that, for some constant $C$ and all $g,f\in G$ and $x,y\in M$, we have
$$
d\big(\alpha(g,x), \alpha(g, y)\big) \leqslant d(x,y)+C
$$
and 
$$
d\big(\alpha(f, \alpha(g,x)), \alpha(fg, x)\big) \leqslant C.
$$
\end{defi}
Thus, if for $g\in G$ we let $\alpha(g)$ denote the map $\alpha(g, \cdot)\colon M\til M$, we see firstly that each $\alpha(g)$ is a contraction $M\til M$ up to an additive defect $C$ and secondly that $\alpha(f)\alpha(g)$ and $\alpha(fg)$ agree as maps on $M$ up to the same additive defect $C$. Thus, a nearly isometric action needs neither be a true action nor do the maps need to be exact isometries. 

Nearly isometric actions can be constructed from Banach space valued quasi-cocycles.
\begin{defi}
A map $b\colon G\til E$ from a group $G$ to a Banach space $E$ is said to be a {\em quasi-cocycle} provided that there is an isometric linear representation $\pi\colon G\curvearrowright E$ so that
$$
\sup_{g,f\in G}\Norm{\pi(g)b(f)+b(g)-b(gf)}<\infty.
$$
\end{defi}
Observe that if $b\colon G\til E$ is a quasi-cocycle associated to an isometric linear representation $\pi\colon G\curvearrowright E$, then the formula
$$
\alpha(g)x= \pi(g)x+b(g)
$$
defines a nearly isometric action of $G$ on $E$. In fact, in this case, each map $\alpha(g)$ is an actual isometry. 
We also have the following converse to this.

\begin{lemme}
Suppose that $\alpha\colon G\times E\til E$ is a nearly isometric action of a group $G$ on a Banach space $E$  so that the $\alpha(g)\colon E\til E$ are uniformly cobounded. Then $b\colon G\til E$ defined by $b(g)=\alpha(g)0$ is a quasi-cocycle associated to an isometric linear representation $\pi\colon G\curvearrowright E$ so that
$$
\sup_{g\in G}\sup_{x\in E}\norm{\pi(g)x+b(g)-\alpha(g)x}<\infty.
$$
\end{lemme}

\begin{proof}
Since the maps $\alpha(g)$ are uniformly cobounded, there is a constant $C$ larger than the defect of $\alpha$ so that $\inf_{y\in E}\norm{x-\alpha(g)y}< C$ for all $g\in G$ and $x\in E$. 

Given any $x\in E$, find $y\in E$ with $\norm{x-\alpha(1)y}\leqslant C$ and observe that
\[\begin{split}
\norm{\alpha(1)x-x}
&\leqslant \norm{\alpha(1)x-\alpha(1)y}+C\\
&\leqslant \norm{\alpha(1)x-\alpha(1)\alpha(1)y}+2C\\
&\leqslant \norm{x-\alpha(1)y}+3C\\
&\leqslant 4C.
\end{split}\]
Thus, for all $x,y\in E$ and $g\in G$, 
\[\begin{split}
\norm{x-y}
&\leqslant \norm{\alpha(1)x-\alpha(1)y}+8C\\
&\leqslant \norm{\alpha(g\inv)\alpha(g)x-\alpha(g\inv)\alpha(g)y}+10C\\
&\leqslant \norm{\alpha(g)x-\alpha(g)y}+11C\\
&\leqslant \norm{x-y}+12C,
\end{split}\]
showing that $\alpha(g)$ is a near isometry of $E$ with defect $11C$.  Hence, as observed in Section \ref{section solvent},  there is a linear isometry $\pi(g)$ of $E$ so that, for $b(g)=\alpha(g)0$ and $K=5000C$, we have
$$
\sup_{x\in E}\norm{\pi(g)x+b(g)-\alpha(g)x}\leqslant K.
$$
It follows that the map $A\colon G\til {\rm Aff}(E)$ given by $A(g)x=\pi(g)x+b(g)$ defines a {\em rough action} of $G$ on $E$ in the sense of Section 13.1 \cite{monod2}, that is, each $A(g)$ is an affine isometry of $E$ and  $\sup_{x\in E}\norm{A(g)A(f)x-A(gf)x}\leqslant 4K<\infty$. In particular, $\pi\colon G\til {\rm Isom}(E)$ is an isometric linear representation of $G$ and $b\colon G\til E$ an associated quasi-cocycle (Lemma 13.1.2 \cite{monod2}). Indeed $\norm{\pi(g)b(f)+b(g)-b(gf)}=\norm{A(g)A(f)0-A(gf)0}\leqslant 4K$, whereby
\[\begin{split}
\norm{\pi(g)\pi(f)x-&\pi(gf)x}\\
&\leqslant\Norm{\pi(g)\big(\pi(f)x+b(f)\big)+b(g)-\big(\pi(gf)x+b(gf)\big)}+4K\\
&\leqslant\Norm{A(g)A(f)x-A(gf)x}+4K\\
&\leqslant 8K
\end{split}\]
for all $x\in E$. As $\norm{\pi(g)\pi(f)x-\pi(gf)x}$ is positive homogeneous in $x$, we see that $\pi(g)\pi(f)=\pi(gf)$ for all $g,f\in G$.
\end{proof}

As our applications deal with topological groups, it is natural to demand that a nearly isometric action $\alpha$ respects some of the topological group structure. As continuity may be too restrictive, a weaker assumption is that some orbit map is bornologous.

\begin{lemme}\label{baire orbit map}
Let $\alpha\colon G\times M\til M$ be a nearly isometric action of a Polish group. Then, if some orbit map $g\in G\mapsto \alpha(g)x\in M$ is Baire measurable, it is also bornologous.
\end{lemme}

\begin{proof}
For every $n$, let $B_n=\{g\in G\del d\big(\alpha(g)x,x\big)<n\;\&\; d\big(\alpha(g\inv)x,x\big)<n\}$. Then $G=\bigcup_nB_n$ is a covering of $G$ by countable many Baire measurable symmetric subsets, whence by the Baire category theorem some $B_n$ must be nonmeagre. Applying Pettis' theorem, it follows that $B_nB_n$ is an identity neighbourhood in $G$.

Now suppose that $A\subseteq G$ is relatively (OB) and find a finite set $F\subseteq G$ and $m\geqslant 1$ so that $A\subseteq (FB_nB_n)^m$. Let also $K$ be a number larger than all of  $n$, $\max_{f\in F}d\big(\alpha(f)x,x\big)$ and the defect of $\alpha$. Then, if $h,g\in G$ and $h\inv g\in A$, there are $k_1,\ldots, k_{2m}\in B_n$ and $f_1,\ldots, f_m\in F$ with $g=hf_1k_1k_2h_2k_3k_4\cdots f_mk_{2m-1}k_{2m}$. 
Thus, 
\[\begin{split}
d\big(\alpha(g)x,\alpha(h)x\big)
= & d\big(\alpha(hf_1k_1k_2\cdots f_mk_{2m-1}k_{2m})x,\alpha(h)x\big)\\
\leqslant & d\big(\alpha(hf_1k_1k_2\cdots f_mk_{2m-1}k_{2m})x,\alpha(h)\alpha(f_1k_1k_2\cdots f_mk_{2m-1}k_{2m})x\big)\\
&+d\big(\alpha(h)\alpha(f_1k_1k_2\cdots f_mk_{2m-1}k_{2m})x,\alpha(h)x\big)\\
\leqslant & d\big(\alpha(f_1k_1k_2\cdots f_mk_{2m-1}k_{2m})x,x\big)+2K\\
\leqslant & d\big(\alpha(f_1k_1k_2\cdots f_mk_{2m-1}k_{2m})x,\alpha(f_1)\alpha(k_1k_2\cdots f_mk_{2m-1}k_{2m})x\big)\\
&+ d\big(\alpha(f_1)\alpha(k_1k_2\cdots f_mk_{2m-1}k_{2m})x,\alpha(f)x\big)+ d\big(\alpha(f)x,x\big)+2K\\
\leqslant & d\big(\alpha(k_1k_2\cdots f_mk_{2m-1}k_{2m})x,x\big)+5K\\
\leqslant &\ldots\\
\leqslant & (2+3m)K.
\end{split}\]
I.e., 
$$
h\inv g \in A\;\saa \; d\big(\alpha(g)x,\alpha(h)x\big)\leqslant (2+3m)K,
$$ 
showing that the orbit map $g\in G\mapsto \alpha(g)x\in M$ is bornologous.
\end{proof}

In particular, since a quasi-cocycle $b\colon G\til E$ associated to an isometric linear representation $\pi\colon G\curvearrowright E$ is simply the orbit map $g\mapsto \alpha(g)0$ of the associated rough action $\alpha(g)x=\pi(g)x+b(g)$, we have the following corollary. 

\begin{lemme}\label{baire cocycle}
Let $b\colon G\til E$ be a Baire measurable quasi-cocycle on a Polish group $G$. Then $b$ is bornologous.
\end{lemme}

\begin{exa}\label{automatic borno}
In a great number of mainly topological transformation groups $G$, we can remove the condition of Baire measurability from Lemmas \ref{baire orbit map} and \ref{baire cocycle}. Indeed, in the above proof, we need only that, whenever $G=\bigcup_nB_n$ is a countable increasing covering by symmetric subsets, then there are  $n$ and $m$ so that ${\rm int}(B_n^m)\neq \emptyset$. One general family of groups satisfying this criterion are Polish groups with {\em ample generics}, i.e., so that the diagonal conjugacy action $G\curvearrowright G^n$ has a comeagre orbit for every $n\geqslant 1$ \cite{turbulence}. For example, both ${\rm Isom}(\Q\U)$ and, by the same proof, ${\rm Isom}(\Z\Q)$ have ample generics, while  ${\rm Aut}(T_{\aleph_0})$ has an open subgroup with ample generics \cite{turbulence}. Thus, in all three cases, every quasi-cocycle defined on the group is bornologous.
\end{exa}

\begin{prop}\label{near isom}
Suppose $G\curvearrowright X$ is a geometric Gelfand pair and that $\alpha$ is a nearly isometric action of $G$ on a metric space $M$ so that every bornologous map $X\til M$ is insolvent. If some orbit map $g\in G\mapsto \alpha(g)\xi\in M$ is bornologous, then every orbit $\alpha[G]\zeta$ is bounded.  
\end{prop}

\begin{proof}
Fix  $\xi \in M$ so that the orbit map $g\in G\mapsto \alpha(g)\xi$ is bornologous and let $C$ be the defect of $\alpha$, i.e., for all $g,f\in G$ and $\zeta, \eta\in M$,  
$$
d_X\big(\alpha(g)\zeta, \alpha(g) \eta\big) \leqslant d_X(\zeta,\eta)+C \;\;\;\;\text{ and }\;\;\;\; 
d_X\big(\alpha(f)\alpha(g)\zeta, \alpha(fg)\zeta\big) \leqslant C.
$$

Let also $K$ be a constant witnessing that  $G\curvearrowright X$ is a geometric Gelfand pair. Fix $x\in X$ and choose for each $y\in X$ some $\gamma_y\in G$ such that $d_X\big(\gamma_y(x),y\big)<K$. Define then $\phi\colon X\til M$ by $\phi(y)=\alpha(\gamma_y)\xi$. We set $V=\{g\in G\del d_X(gx,x)<3K\}$, which is a relatively (OB) symmetric identity neighbourhood in $G$ and let 
$$
L=\sup_{f\inv g\in V}d_M\big(\alpha(g)\xi,\alpha(f)\xi\big).
$$

We claim that $\phi$ is bornologous. Indeed, observe that the composition $y\mapsto \gamma_y(x)$ of  $\gamma\colon X\til G$ with the orbit map $g\in G\mapsto g(x)\in X$ will be close to the identity map on $X$. So, as the orbit map $g\mapsto g(x)$ is coarsely proper and hence expanding, this implies that $\gamma$ and thus also $\phi$ are bornologous

Now, suppose that $m\geqslant 1$ and that 
$$
d_M\big(\phi(x),\phi(y)\big)\geqslant m+3C+2L+\theta_\phi(m+2K)
$$ 
for some $y\in X$. We claim that, for all $z,u\in X$, 
$$
d_X(x,y)\leqslant d_X(z,u)\leqslant d_X(x,y)+m \;\;\saa\;\; d_M\big(\phi(z),\phi(u)\big)\geqslant m.
$$

To see this, suppose that $d_X(x,y)\leqslant d_X(z,u)\leqslant d_X(x,y)+m$ and find some $g\in G$ so that $d_X\big(g(x),z\big)<K$ and $d_X\big(z,g(y)\big)+d_X\big(g(y),u\big)<d_X(z,u)+K$, i.e., 
\[\begin{split}
d_X\big(g(y),u\big)
&\leqslant d_X(z,u)+K-d_X\big(z,g(y)\big)\\
&\leqslant d_X(x,y)+m+K-d_X\big(g(x),g(y)\big)+K\\
&=m+2K.
\end{split}\]
In particular, $d_M\big(\phi(gy),\phi(u)\big)\leqslant \theta_\phi(m+2K)$.

Now, $d_X(g\gamma_x(x),\gamma_z(x))\leqslant d_X(g(x),z)+2K<3K$, so $\gamma_z\inv g\gamma_x\in V$, whence
$$
d_M\big(\alpha(g\gamma_x)\xi, \phi(z)\big)=d_M\big(\alpha(g\gamma_x)\xi, \alpha(\gamma_z)\xi\big)\leqslant L.
$$
Similarly, $d_X\big(g\gamma_y(x),\gamma_{gy}(x)\big)\leqslant d_X\big(g\gamma_y(x),g(y)\big)+d_X\big(g(y),\gamma_{gy}(x)\big)<2K$, so $\gamma_{gy}\inv g\gamma_y\in V$, whence
$$
d_M\big(\alpha(g\gamma_y)\xi, \phi(gy)\big)=d_M\big(\alpha(g\gamma_x)\xi, \alpha(\gamma_{gy})\xi\big)\leqslant L.
$$
Thus
\[\begin{split}
d_M\big(\phi(x),\phi(y)\big) 
&=d_M\big(\alpha(\gamma_x)\xi, \alpha(\gamma_y)\xi\big)\\
&\leqslant d_M\big(\alpha(g\inv )\alpha(g\gamma_x)\xi, \alpha(g\inv )\alpha(g\gamma_y)\xi\big)+2C\\
&\leqslant d_M\big(\alpha(g\gamma_x)\xi, \alpha(g\gamma_y)\xi\big)+3C\\
&\leqslant d_M\big(\phi(z),\phi(gy)\big)+3C+2L\\
&\leqslant d_M\big(\phi(z),\phi(u)\big)+3C+2L+\theta_\phi(m+2K),
\end{split}\]
whence $d_M\big(\phi(z),\phi(u)\big)\geqslant m$, proving the claim.

It follows from our claim that, if $\phi$ is unbounded, then, for every $m$, there is some $R_m$ so that
$$
R_m\leqslant d_X(z,u)\leqslant R_m+m \;\;\saa\;\; d_M\big(\phi(z),\phi(u)\big)\geqslant m
$$
i.e., $\phi$ is solvent, which is absurd. So $\phi$ is a bounded map, which easily implies that every orbit $\alpha[G]\zeta$ is bounded in $M$.
\end{proof}

\begin{prop}
Suppose $G\curvearrowright X$ is a geometric Gelfand pair and $H$ a topological group so that every bornologous map $X\til H$ is insolvent.  Then, if $\pi\colon G\til H$ is a continuous homomorphism, $\pi[G]$ is relatively (OB) in $H$.
\end{prop}

\begin{proof}
Pick a constant $K$ witnessing that $G\curvearrowright X$ is a geometric Gelfand pair, fix $x\in X$  and choose for every $y\in X$ some $\gamma_y\in G$ so that $d(\gamma_y(x), y)<K$. As in the proof of Proposition \ref{near isom}, $\gamma\colon X\til G$ is bornologous. Finally, let $V=\{g\in G\del d_X(gx,x)<3K\}$, which is a symmetric relatively (OB) identity neighbourhood in $G$.

Now assume that $A\subseteq H$ is relatively (OB), $m\geqslant 1$ and fix a symmetric relatively (OB) set $D\subseteq G$ so that $\gamma_v\in \gamma_wD$ whenever $d(w,v)\leqslant m+2K$. Suppose that $\gamma_y\notin \gamma_xV\pi\inv(A)DV$ for some $y\in X$. 

Suppose that  $d_X(x,y)\leqslant d_X(z,u)\leqslant d_X(x,y)+m$ for some $z,u\in Z$ and find some $g\in G$ so that $d_X\big(g(x),z\big)<K$ and $d_X\big(z,g(y)\big)+d_X\big(g(y),u\big)<d_X(z,u)+K$, i.e., $d_X\big(g(y),u\big)\leqslant m+2K$. In particular, $\gamma_{gy}\in \gamma_uD$.
Also $d_X(g\gamma_x(x),\gamma_z(x))\leqslant d_X(g(x),z)+2K<3K$ and 
$$
d_X\big(g\gamma_y(x),\gamma_{gy}(x)\big)\leqslant d_X\big(g\gamma_y(x),g(y)\big)+d_X\big(g(y),\gamma_{gy}(x)\big)<2K,
$$ 
so $g\gamma_x\in \gamma_zV$ and  $g\gamma_y\in \gamma_{gy}V$.
It follows that
\[\begin{split}
\gamma_x\inv \gamma_y=(g\gamma_x)\inv g\gamma_y\in V\gamma_{z}\inv\gamma_{gy}V\subseteq V\gamma_{z}\inv \gamma_uDV
\end{split}\]
and so $\pi(\gamma_{z})\inv \pi(\gamma_u)\notin A$. In other words, for all $z,u\in X$, 
$$
d_X(x,y)\leqslant d_X(z,u)\leqslant d_X(x,y)+m \;\; \saa\;\; \pi(\gamma_{z})\inv \pi(\gamma_u)\notin A.
$$

This shows that, if, for all relatively (OB) sets $A\subseteq H$ and $D\subseteq G$, there is some $\gamma_y\notin \gamma_xV\pi\inv (A)DV$, then the map $z\in X\mapsto \pi(\gamma_z)\in H$ is both bornologous and solvent, which is impossible. So choose some $A$ and $D$ for which it fails.
As ${\rm im}(\gamma)$ is cobounded in $G$, it follows that $G=V\pi\inv(A)U$ for some relatively (OB) set $U\subseteq G$ and hence $\pi[G]$ is included in the relatively (OB) set  $\pi[V]A\pi[U]$. In particular, $\pi[G]$ is relatively (OB) in $H$.
\end{proof}

We are now in a position of deducing  the main application of this section.  Note first that  ${\rm Isom}(\Q\U)$ is isomorphic to a closed subgroup of the group $S_\infty$ of all permutations of a countable set and so, in particular, ${\rm Isom}(\Q\U)$ admits continuous unitary representations given by permutations of an othogonal basis $(e_x)_{x\in \Q\U}$. Moreover, by results of \cite{solecki} and \cite{turbulence}, ${\rm Isom}(\Q\U)$ is approximately compact and thus amenable. As shown in \cite{autom}, ${\rm Isom}(\Q\U)$ is coarsely equivalent to $\Q\U$ itself. So, as $c_0$ nearly isometrically embeds into $\Q\U$, this shows that ${\rm Isom}(\Q\U)$ cannot have a coarsely proper affine isometric action on a reflexive space. This is in opposition to the result of Brown--Guentner \cite{BG} and Haagerup--Przybyszewska \cite{haagerup-affine} that every locally compact second countable group admits such an action.  However, our result here indicates a much higher degree of geometric incompatibility.

\begin{thm}\label{fix}
Let $E$ be either a reflexive Banach space or $E=L^1([0,1])$ and let $G$ be ${\rm Isom}(\Z\U)$ or ${\rm Isom}(\Q\U)$. Then every quasi-cocycle $b\colon G\til E$ is bounded. In particular, every affine isometric action $G\curvearrowright E$  has a fixed point.
\end{thm}

\begin{proof}
Observe that there are near isometries from $c_0$ into both $\Z\U$ and $\Q\U$. It thus follows from Theorem \ref{kalton2}, respectively Corollary \ref{kalton3}, that every bornologous map from $\Z\U$ or $\Q\U$ into a reflexive space or into $L^1([0,1])$ is insolvent.  Now, by Example \ref{automatic borno}, every quasi-cocycle on $G$ is bornologous, so it follows from Proposition \ref{near isom} that every quasi-cocycle $b\colon G\til E$ is bounded. 

Now, if $\alpha\colon G\curvearrowright E$ is an affine isometric action, then the associated cocycle in bounded and thus the affine action has a bounded orbit. 
Thus, by the Ryll-Nardzewski fixed point theorem \cite{ryll} for the reflexive case or the fixed point theorem of U. Bader, T. Gelander and N. Monod \cite{monod} for $E=L^1([0,1])$, there is a fixed point in $E$.
\end{proof}

Theorem \ref{fix} is motivated by questions pertaining to the characterisation of property (OB) among non-Archimedean Polish groups. Every continuous affine isometric action of a Polish group with property (OB) on a reflexive space has a fixed point and one may ask if this characterises property (OB) for {\em non-Archimedean} Polish groups, i.e., closed subgroups of $S_\infty$. As ${\rm Isom}(\Q\U)$ acts transitively on the infinite-diameter metric space $\Q\U$, it fails property (OB) and so the answer to our question is no. However, the following question remains open.
\begin{quest}
Suppose $\alpha\colon {\rm Isom}(\Q\U)\curvearrowright E$ is a continuous affine action on a reflexive Banach space $E$. Does $\alpha$ necessarily have a fixed point?
\end{quest}
The difference with Theorem \ref{fix} here is that the action is not required to be isometric and thus the linear part $\pi$ could map onto an unbounded subgroup of the general linear group ${\rm GL}(E)$.



\end{document}